\documentclass[11pt,leqno]{amsart}

\usepackage[colorlinks=false,hidelinks]{hyperref} 
\usepackage[T1,T5]{fontenc} 
\usepackage[vietnamese,english]{babel}
\usepackage{amssymb}
\usepackage{amsrefs}
\usepackage[shortlabels]{enumitem}
\usepackage{listings}
\usepackage{longtable}

\newtheorem{thm}{Theorem}[section]
\newtheorem{lemma}[thm]{Lemma}
\newtheorem{cor}[thm]{Corollary}
\newtheorem{prop}[thm]{Proposition}
\newtheorem{alg}[thm]{Algorithm}
\newtheorem{conj}[thm]{Conjecture}

\theoremstyle{definition}
\newtheorem{definition}[thm]{Definition}
\newtheorem{example}[thm]{Example}
\newtheorem{prob}[thm]{Problem}

\theoremstyle{remark}
\newtheorem{rmk}[thm]{Remark}
\newtheorem{obs}[thm]{Observation}
\newtheorem*{notation}{Notation}

\numberwithin{equation}{section}

\newcommand{\op}{^{\mathrm{op}}}
\DeclareMathOperator{\Anti}{Anti}
\DeclareMathOperator{\Good}{Good}
\DeclareMathOperator{\id}{id}
\newcommand{\Z}{\mathbb{Z}}
\newcommand{\inv}{^{-1}}
\DeclareMathOperator{\Aut}{Aut}
\DeclareMathOperator{\Inn}{Inn}
\DeclareMathOperator{\Conj}{Conj}
\DeclareMathOperator{\Alex}{GAlex}
\newcommand{\F}{\mathcal{F}}
\newcommand{\G}{\mathcal{G}}
\newcommand{\oo}{\mathcal{O}}
\renewcommand{\phi}{\varphi}
\newcommand{\surj}{\twoheadrightarrow}
\newcommand{\bij}{\xrightarrow{\sim}}
\newcommand{\SL}{\mathrm{SL}}
\newcommand{\SU}{\mathrm{SU}}
\newcommand{\Rc}{\mathsf{Rack}}
\newcommand{\glr}{\mathsf{GLR}}
\newcommand{\LR}{\mathsf{LR}}
\newcommand{\uu}{\mathtt{u}}
\newcommand{\LQ}{\mathsf{LQ}}
\newcommand{\kir}{\mathsf{KIR}}
\newcommand{\LK}{\mathsf{LK}}
\newcommand{\InvRk}{\mathsf{InvRack}}
\newcommand{\SK}{\mathsf{SK}}
\DeclareMathOperator{\Core}{Core}
\DeclareMathOperator{\Dic}{Dic}

\makeatletter
\@namedef{subjclassname@2020}{\textup{2020} Mathematics Subject Classification}
\makeatother

\makeatletter
\newcommand{\addresseshere}{%
	\enddoc@text\let\enddoc@text\relax
}
\makeatother

\hypersetup{
	pdflang={en-US},
	pdftitle={Good involutions of conjugation subquandles},
	pdfauthor={Lực Ta},
	pdfsubject={Mathematics},
	pdfkeywords={Classification, conjugation, core quandle, enumeration, good involution, group, quandle, rack, symmetric quandle}
}

\begin{document}
	
	\title[
    Good involutions of conjugation subquandles
    ]{
    Good involutions of conjugation subquandles
    }
	\author{L\d\uhorn c Ta}	
	
	\address{Department of Mathematics, University of Pittsburgh, Pittsburgh, Pennsylvania 15260}
	\email{ldt37@pitt.edu}
	
	\subjclass[2020]{Primary 20N02; Secondary 20E34, 20E45, 57K12}
	
	\keywords{Classification, conjugation, core quandle, enumeration, good involution, group, quandle, rack, symmetric quandle}
	
	\begin{abstract}
	Posed by Taniguchi, the classification of quandles with good involutions is a difficult question with applications to surface-knot theory. We address this question for subquandles of conjugation quandles, including all core quandles. We also study good involutions of faithful racks. In particular, we obtain sharp bounds on the number of good involutions of racks in these families.

As an application of our results, we implement group-theoretic algorithms that compute all good involutions of conjugation quandles and core quandles; we provide data for those up to order $23$. As another application, we construct infinite families of connected, noninvolutory symmetric quandles.

We also classify symmetric and Legendrian racks, quandles, and kei up to order $8$ using a computer search. Finally, we exhibit an equivalence of categories between racks and Legendrian racks that induces an equivalence between involutory racks, Legendrian kei, and symmetric kei.
	\end{abstract}
	\maketitle
	
\section{Introduction}\label{sec:intro}

Introduced by Kamada \cite{symm-quandles-2} in 2007, \emph{symmetric racks} are nonassociative algebraic structures used to study surface-links \citelist{\cite{symm-quandles-2}\cite{tensor}\cite{yasuda}\cite{dihedral}}, handlebody-links \cite{multiple}, $3$-manifolds \cite{nosaka}, compact surfaces with
boundary in ribbon forms \cite{symm-racks}, and framed virtual knots in thickened surfaces \cite{virtual}. Defined as pairs $(R,\rho)$ where $R=(X,s)$ is a rack and $\rho$ is a function called a \emph{good involution} of $R$, symmetric racks enjoy various (co)homology theories \citelist{\cite{symm-quandles-2}\cite{karmakar}} and connections to group theory \citelist{\cite{lie}\cite{multiple}}. 

Symmetric racks can be traced back to \emph{racks}, which Fenn and Rourke \cite{fenn} introduced in 1992 to construct complete invariants of framed links in connected $3$-manifolds; \emph{quandles}, which Joyce~\cite{joyce} and Matveev~\cite{matveev} independently introduced in 1982 to construct complete invariants of unframed links; and \emph{kei}, which Takasaki~\cite{takasaki} introduced in 1943 to study Riemannian symmetric spaces. See \citelist{\cite{quandlebook}\cite{book}} for references on the theory and \cite{survey} for a survey of the algebraic state of the art.

In this paper, we study the following classification problems posed by Taniguchi \cite{taniguchi}.

\begin{prob}\label{prob1}
    Find necessary and sufficient conditions for good involutions $\rho$ of quandles to exist.
\end{prob}

\begin{prob}\label{prob2}
    Given a quandle $R$, classify all symmetric quandles of the form $(R,\rho)$.
\end{prob}

Although various constructions of symmetric quandles exist \citelist{\cite{galkin-symm}\cite{dihedral}\cite{symm-quandles}}, few complete classification results in the vein of Problems \ref{prob1} and \ref{prob2} exist in the literature. 
In 2010, Kamada and Oshiro \cite{symm-quandles-2} addressed Problem \ref{prob1} for kei and Problem \ref{prob2} for dihedral quandles and trivial quandles. 
In 2023, Taniguchi \cite{taniguchi} solved Problem \ref{prob1} for generalized Alexander quandles and Problem \ref{prob2} for connected generalized Alexander quandles. 

This paper addresses Problems \ref{prob1} and \ref{prob2} for large classes of quandles of broad interest called \emph{conjugation quandles} and \emph{core quandles}. These quandles are constructed using the functors $\Conj$ and $\Core$ from the category of groups to the category of quandles. Joyce \cite{joyce} and Matveev \cite{matveev} independently introduced these quandles to generalize Wirtinger presentations of knot groups and develop complete invariants of knots. 

Conjugation quandles and their subquandles have significant applications to low-dimensional topology \citelist{\cite{eisermann}\cite{fenn}\cite{joyce}}, group theory \citelist{\cite{lattice}\cite{conjugation}\cite{lie}}, and algebraic combinatorics \citelist{\cite{qualgebra}\cite{taGQ}}. In particular, good involutions of conjugation quandles are used to study minimal triple point numbers of nonorientable surface-knots \cite{dihedral}, coloring invariants of handlebody-knots \cite{multiple}, and noncommutative geometry applied to finite groups \cite{lie}. Core quandles \cite{core} and many other classes of quandles (see \cite{embed} for a list) canonically embed into conjugation quandles; our results apply to all such quandles. 

\subsection{Overview of results}
Our main results are the following. Theorem \ref{thm2} states that good involutions $\rho$ of conjugation subquandles $\Conj X$ of conjugation quandles $\Conj G$ factor through the natural projection $\pi:X\surj X/\langle X\rangle$ and a unique function $\zeta^*:X/\langle X\rangle\to Z(\langle X\rangle)$, where $X/\langle X\rangle$ denotes the set of orbits of $X$ under the action of the subgroup $\langle X\rangle\leq G$ by conjugation. Moreover, $\zeta^*$ completely determines $\rho$. Theorem \ref{thm1} gives an analogous statement for core quandles $\Core G$.

\begin{thm}\label{thm2}
    Let $G$ be a group, and let $\Conj X$ be a subquandle of $\Conj G=(G,s)$. 
    Then for all functions $\rho:X\to X$, $\rho$ is a good involution of $\Conj X$ if and only if there exists a function $\zeta:X\to Z(\langle X\rangle)$ such that
    \[
    \rho(x)=(\zeta(x)x)\inv,\qquad\zeta s_x=\zeta=\zeta\rho
    \]
    for all $x\in X$. (See Theorem \ref{thm:conj}; cf.\ Remark \ref{obs:cf} and Corollary \ref{cor:connected}.)
\end{thm}

\begin{thm}\label{thm1}
    Let $G$ be a group such that the core quandle $\Core G=(G,s)$ is nontrivial, and let $T$ be the subgroup
    \[
    T:=\{z\in Z(G)\mid z^2=1\}
    \]
    of the center $Z(G)$. 
    Then for all functions $\rho:G\to G$, $\rho$ is a good involution of $\Core G$ if and only if there exists a function $\alpha:G\to T$ such that
    \[
    \rho(g)=\alpha(g)g,\qquad \alpha s_g=\alpha=\alpha\rho
    \]
    for all $g\in G$. (See Theorem \ref{thm:core}; cf.\ Remark \ref{rmk:core} and Corollaries \ref{cor:core} and \ref{cor:taka}.)
\end{thm}

For example, it is well-known that if $\Conj X$ is a subquandle of $\Conj G$ and $X$ is closed under inverses, then the inversion map $x\mapsto x\inv$ is a good involution of $\Conj X$. This good involution corresponds to letting $\zeta:X\to Z(\langle X\rangle)$ be the constant function $x\mapsto 1$.

As applications of Theorem \ref{thm2} and \ref{thm1}, we describe procedures (Algorithms \ref{alg} and \ref{alg2}) for constructing all good involutions of conjugation subquandles and core quandles, and we provide data for groups up to order $23$ in a GitHub repository \cite{code}. This addresses Problem \ref{prob2}.

\subsubsection{Computer search}

\begin{table}
\centering
\caption{Enumeration of isomorphism classes of various types of symmetric racks up to order $8$.}
\label{table:counts}
\begin{tabular}{l|ccccccccc}
Order                        & 0 & 1 & 2 & 3 & 4  & 5  & 6   & 7    & 8     \\ \hline
Symmetric racks           & 1 & 1 & 4 & 9 & 42 & 154 & 1064 & 6678 & 73780\\
Symmetric quandles    & 1 & 1 & 2 & 5 & 13 & 44 & 187 & 937 & 6459\\
Symmetric kei        & 1 & 1 & 2 & 5 & 13 & 42 & 180 & 906 & 6317
\end{tabular}
\end{table}

Through an exhaustive search in the computer algebra system \texttt{GAP} \cite{GAP4}, we tabulated all isomorphism classes of symmetric racks, quandles, and kei up to order $8$. This answers Problem \ref{prob2} for all racks up to order $8$. Our classification used Vojtěchovský and Yang's \cite{library} library of racks in a way similar to the algorithms described in \cite{taGL}*{App.\ A}.  

We provide our code and data in a GitHub repository \cite{code}. Table \ref{table:counts} enumerates our data.

\subsubsection{CNS-quandles}
Connected, noninvolutory symmetric quandles, which we call \emph{CNS-quandles}, are particularly helpful for studying $3$-manifolds \cite{nosaka} and nonorientable surface-knots \cite{dihedral}. However, CNS-quandles are difficult to construct. For example, generalized Alexander quandles have good involutions if and only if they are kei \cite{taniguchi}*{Thm.\ 1.3}, and our computer search showed the following.
\begin{prop}
    Up to isomorphism, the only CNS-quandle of order $8$ or lower is the symmetric Galkin quandle of order $6$ constructed by Clark et al.\ \cite{galkin-symm}*{Prop.\ 5.13}.
\end{prop}
Thus, the following problems arise; see \cite{galkin-symm}*{Sec.\ 5}.

\begin{prob}\label{prob:cns}
    Construct infinite families of CNS-quandles.
\end{prob}

\begin{prob}\label{prob:cns2}
    For which integers $k\in\Z^+$ do there exist CNS-quandles of order $k$?
\end{prob}

Carter et al.\ \cite{dihedral}*{Thm.\ 3.1} constructed CNS-quandles of order $2^m m$ for all odd integers $m\geq 1$, while Clark et al.\ \cite{galkin-symm}*{Prop.\ 5.13} constructed CNS-quandles of order $6n$ for all positive integers $n\geq 1$. To the author's knowledge, these remain the only two answers to Problems \ref{prob:cns} and \ref{prob:cns2} in the literature. 
To address this, we introduce a broad family of CNS-quandles in Proposition \ref{prop:cns}. In particular, two special cases of our construction give answers to Problem \ref{prob:cns2}.

\begin{prop}\label{prop:sl}
    For all integers $m\geq 3$ and powers $q$ of odd primes, and for all integers $n\geq 5$ and $k$ such that $3\leq k \leq 2\lfloor n/2\rfloor$, there exist CNS-quandles of orders
    \[
    \frac{(q^m-1)(q^{m-1}-1)}{q-1},\qquad \frac{n!}{k(n-k)!}.
    \]
    (See Examples \ref{ex:sl} and \ref{ex:an} for details.)
\end{prop}

\begin{rmk}
    If $q=3=k$, then for all integers $m\geq 3$ and $n\geq 5$ such that $n\not\equiv 0,1,2$ mod $9$, the CNS-quandles in Proposition \ref{prop:sl} do not share orders with the CNS-quandles constructed by Clark et al. Most of them also do not share orders with the CNS-quandles constructed by Carter et al.
\end{rmk}

\subsubsection{Connections with Legendrian racks}

Motivated by an observation of Karmakar et al.\ \cite{karmakarGL}*{Rem.\ 3.8} and the classification problems posed in \citelist{\cite{ceniceros}\cite{karmakarGL}\cite{taGL}}, we draw connections between symmetric racks and \emph{Legendrian racks}, which are algebraic structures used to distinguish Legendrian knots in contact three-space.

Generalizing prior work of Kulkarni and Prathamesh \cite{original} in 2017, Ceniceros et al.\ \cite{ceniceros} introduced Legendrian racks in 2021. In 2023, Karmakar et al.\ \cite{karmakarGL} and Kimura \cite{bi} independently introduced generalizations of Legendrian racks called \emph{GL-racks}. In 2025, the author \cite{taGL} studied algebraic aspects of GL-racks and classified them up to order $8$, and Cheng and He \cite{he} showed that GL-racks distinguish Legendrian knots up to the absolute values of their classical invariants.

In this paper, we show that the classification problems for symmetric kei, Legendrian kei, and involutory racks are equivalent. We also obtain similar results for other types of GL-racks.

\begin{thm}\label{thm:legendrian}
    The categories of symmetric kei, Legendrian kei, and involutory racks are equivalent. In fact, good involutions and Legendrian structures coincide for kei. (See Theorem \ref{thm:gl}.)
\end{thm}

In particular, the bottom row of Table \ref{table:counts} enumerates isomorphism classes of Legendrian kei and those of involutory racks; cf.\ Table \ref{table:gl} in Section \ref{sec:leg}.

\subsection{Structure of the paper}
In Section \ref{sec:prelims}, we provide definitions and examples of racks, quandles, and kei. In Section \ref{sec:prelims2}, we do the same for rack antiautomorphisms and good involutions.

In Section \ref{sec:general}, we describe good involutions in terms of rack (anti)automorphisms. In particular, we deduce that only self-dual racks have good involutions. This addresses Problem \ref{prob1}.

In Section \ref{sec:special}, we discuss Problems \ref{prob1} and \ref{prob2} for faithful racks. A class of examples shows that infinitely many quandles have involutory antiautomorphisms but no good involutions.

In Section \ref{sec:conj}, we prove Theorem \ref{thm2} and discuss some special cases along the way. This addresses Problems \ref{prob1} and \ref{prob2} for subquandles of conjugation quandles.

In Section \ref{sec:core}, we prove Theorem \ref{thm1} by reducing to a special case of Theorem \ref{thm2}. This addresses Problems \ref{prob1} and \ref{prob2} for core quandles.

In Section \ref{sec:algos}, we describe procedures for computing all good involutions of conjugation subquandles and core quandles. We provide data for groups up to order $23$ \cite{code}. This addresses Problem \ref{prob2}.

In Section \ref{sec:construct}, we compute all good involutions of several classes of connected quandles. Along the way, we obtain Proposition \ref{prop:sl}. This addresses Problems \ref{prob:cns} and \ref{prob:cns2}. 

In Section \ref{sec:leg}, we construct an equivalence of categories between racks and Legendrian racks. We discuss various restrictions of this equivalence; in particular, we obtain Theorem \ref{thm:legendrian}.

In Section \ref{sec:open}, we pose several questions for further research.

\begin{notation}
    We use the following notation throughout this paper. Denote the composition of functions $\phi:X\to Y$ and $\psi:Y\to Z$ by $\psi\phi$. Given a set $X$, let $\id_X$ denote the identity map of $X$, and let $S_X$ denote the symmetric group of $X$. 
    Given an integer $n\in\Z^+$, let $S_n$ and $A_n$ be the symmetric and alternating groups on $n$ letters, respectively.
\end{notation}

\subsection*{Acknowledgments}
I thank Samantha Pezzimenti, Jose Ceniceros, and Peyton Wood for respectively introducing me to knot theory, quandles, and Legendrian racks. I also thank Sam Raskin for advising me during the writing of \cite{taGL}, the results of which helped inspire Theorem \ref{thm:legendrian}.

\section{Preliminaries: Racks}\label{sec:prelims}

\subsection{Racks and quandles}

\begin{definition}
    Let $X$ be a set, let $s:X\to S_X$ be a map, and write $s_x:=s(x)$ for all elements $x\in X$. We call the pair $(X,s)$ a \emph{rack} if \begin{equation*}
		s_xs_y=s_{s_x(y)}s_x
	\end{equation*}
    for all $x,y\in X$. We call $s$ the \emph{rack structure} of $(X,s)$, and we say that $|X|$ is the \emph{order} of $(X,s)$.
\end{definition}

\begin{rmk}
    Some authors define racks as sets $X$ with equipped with a nonassociative binary operation $\triangleright:X\times X\to X$ satisfying certain axioms, including right-distributivity. These two definitions are equivalent \citelist{\cite{survey} \cite{tetrahedral}\cite{joyce}} via the formula
    \[
    s_y(x)=x\triangleright y.
    \]
\end{rmk}

\begin{rmk}
    Although the empty set can be viewed as a rack, we will only discuss nonempty racks in this paper.
\end{rmk}

\begin{definition}
    Let $R:=(X,s)$ be a rack.
    \begin{itemize}
        \item We say that $R$ is a \emph{quandle} if $s_x(x)=x$ for all $x\in X$.
        \item If $R$ is a quandle and $Y\subseteq X$ is a subset such that $s_y(Y)= Y$ for all $y\in Y$, then we say that the quandle $(Y,s|_Y)$ is a \emph{subquandle} of $R$. 
        \item We say that $R$ is \emph{involutory} if $s_x^2=\id_X$ for all $x\in X$. 
        \item A \emph{kei} is an involutory quandle.
        \item We say that $R$ is \emph{faithful} if $s$ is injective.
    \end{itemize}
\end{definition}

\begin{example}[\cite{fenn}]
    Let $X$ be a set, and fix $\sigma\in S_X$. Define $s:X\rightarrow S_X$ by $x\mapsto \sigma$, so that $s_x=\sigma$ for all $x\in X$. Then $R:=(X,s)$ is a rack called a \emph{permutation rack}, \emph{cyclic rack}, or \emph{constant action rack}.
        
    Note that $R$ is a quandle if and only if $\sigma=\id_X$, in which case we call $R$ a \emph{trivial quandle}. Moreover, $R$ is involutory if and only if $\sigma$ is an involution, and $R$ is faithful if and only if $|X|\leq 1$.
\end{example}

    \begin{example}[\citelist{\cite{joyce}\cite{fenn}\cite{conjugation}}]
    Let $G$ be a group, and define $s:G\rightarrow S_G$ by sending each element $g\in G$ to the conjugation map defined by \[s_g(h):= ghg\inv\]
    for all $h\in G$. 
	Then $\Conj G:=(G,s)$ is a quandle called a \emph{conjugation quandle} or \emph{conjugacy quandle}. Note that $s_g\inv=s_{g\inv}$ for all $g\in G$.
\end{example}

\begin{example}[\citelist{\cite{joyce}\cite{core}\cite{core2}}]\label{ex:core}
    Let $G$ be a group, and define $s:G\to S_G$ by
    \[
    s_g(h):=gh\inv g
    \]
    for all $g,h\in G$. Then $\Core G:=(G,s)$ is a kei called the \emph{core quandle} of $G$. 
    
    Note that $\Core G$ is faithful if and only if the center $Z(G)$ is $2$-torsionless \cite{core2}*{Lem.\ 4.1}. Moreover, $\Core G$ is a permutation rack if and only if $G$ is an elementary abelian $2$-group, in which case $\Core G$ is a trivial quandle.

    Furthermore, if $G$ is an additive abelian group, then the core quandle $T(G):=\Core(G)$ is called a \emph{Takasaki kei} \cite{quandlebook}*{Ex.\ 54}. Note that $T(G)$ is faithful if and only if $G$ is $2$-torsionless. 
    
    In particular, if $G$ is cyclic of order $n$, then we call $R_n:=T(\Z/n\Z)$ a \emph{dihedral quandle}. By the above discussion, $R_n$ is faithful if and only if $n$ is odd.
\end{example}

\begin{example}[\cite{taniguchi}, cf.\ \cite{alex}]\label{ex:alex}
    Let $G$ be a group, and let $\phi\in\Aut G$ be an automorphism of $G$. The \emph{generalized Alexander quandle} $\Alex(G,\phi)$ is defined to be the quandle $(G,s)$, where
    \[
    s_g(h):=\phi(hg\inv)g
    \]
    for all $g,h\in G$. A routine check shows that $\Alex(G,\phi)$ is faithful if and only if the identity element of $G$ is the only fixed point of $\phi$.

    In particular, let $G$ be an additive abelian group, and let $\phi$ be the inversion map $g\mapsto -g$. Then $\Alex(G,\phi)$ equals the Takasaki kei $T(G)$. 
\end{example}

\begin{example}[\cite{tetrahedral}*{Ex.\ 2.8}]\label{ex:tetra}
    Define the set $X:=\{1,2,3,4\}$, and define $s:X\to S_4$ by
    \[
    s_1:=(234),\quad s_2:=(143),\quad s_3:=(124),\quad s_4:=(132)
    \]
    in cycle notation. Then $(X,s)$ is a faithful, noninvolutory quandle called the \emph{tetrahedral quandle} or \emph{regular tetrahedron quandle}.
\end{example}

\subsubsection{Conjugation subquandles}

We discuss subquandles of conjugation quandles. In the following, let $G$ be a group, and let $s:G\to S_G$ denote the rack structure of $\Conj G$.

\begin{lemma}[\cite{conjugation}*{Thm.\ 3.1}]\label{lem:conj}
    For all subsets $X\subseteq G$, the pair $\Conj X:=(X,s|_X)$ is a subquandle of $\Conj G$ if and only if $X$ is closed under conjugation by elements of the subgroup $\langle X\rangle$ of $G$ generated by $X$. 
\end{lemma}

\begin{obs}\label{obs:conj2}
    Let $\Conj X$ be a subquandle of $\Conj G$. Then $\Conj X$ is a kei if and only if $x^2\in Z(\langle X\rangle)$ for all $x\in X$, and $\Conj X$ is faithful if and only if the action of $\langle X\rangle$ on $X$ by conjugation is faithful. In particular, $\Conj G$ is faithful if and only if $G$ is centerless.
\end{obs}

\begin{example}
    The conjugation quandle of the quaternion group $Q_8$ is a nonfaithful kei.
\end{example}

Of particular interest in this paper are subquandles $\Conj X$ of conjugation quandles $\Conj G$ whose underlying sets $X$ are symmetric subsets of $G$.

\begin{definition}
    Let $\Conj X$ be a subquandle of $\Conj G$. We say that $\Conj X$ is \emph{inversion-closed}\footnote{We chose this name over the name ``symmetric subquandle'' to avoid confusion with Definition \ref{def:symmetric}.} if $x\inv\in X$ for all $x\in X$.
\end{definition}

\begin{rmk}
    By Lemma \ref{lem:conj}, a subquandle $\Conj X$ of $\Conj G$ is inversion-closed if and only if every element $x\in X$ is conjugate to its inverse $x\inv$ in the subgroup $\langle X\rangle\leq G$.
\end{rmk}

\begin{example}
    Let $n\in\Z^+$ be a positive integer, and let $X$ be any conjugacy class in the symmetric group $S_n$. Then $\Conj X$ is an inversion-closed subquandle of $\Conj S_n$.
\end{example}

\begin{example}\label{ex:a4}
    Let $X$ be the split conjugacy class
    \[
    X:=\{(234),(143),(124),(132)\}
    \]
    in the alternating group $A_4$. Then $\langle X\rangle=A_4$ is centerless, so $\Conj X$ is a faithful, noninvolutory subquandle of $\Conj A_4$. Moreover, $\Conj X$ is not inversion-closed. 
\end{example}

\subsubsection{Rack homomorphisms}
Racks form a category $\Rc$ with the following morphisms; quandles form a full subcategory of this category. See \citelist{\cite{rack-roll}\cite{center}\cite{joyce}} for further discussion of racks and quandles from a categorical perspective.

\begin{definition}
    Given two racks $(X,s)$ and $(Y,t)$, we say that a map $\varphi:X\to Y$ is a \emph{rack homomorphism} if \[\varphi s_x=t_{\varphi(x)}\varphi\] for all $x\in X$. A \emph{rack isomorphism} is a bijective rack homomorphism, and a \emph{rack automorphism} of is an isomorphism from a rack to itself.
\end{definition}

\begin{example}[\citelist{\cite{joyce}\cite{matveev}}]
    It is straightforward to check that the assignments $G\mapsto \Conj G$ and $G\mapsto \Core G$ define faithful functors from the category of groups to the category of quandles.
\end{example}

\begin{example}
    The tetrahedral quandle from Example \ref{ex:tetra} is isomorphic to the quandle in Example~\ref{ex:a4}.
\end{example}

\begin{definition}
    We denote the \emph{automorphism group} of a rack $R=(X,s)$ by $\Aut R$. The \emph{inner automorphism group}\footnote{Some authors call $\Inn R$ the \emph{right multiplication group} $\operatorname{RMlt}R$ or \emph{operator group} $\operatorname{Op}R$ of $R$.} $\Inn R$ of $R$ is the subgroup of $S_X$ generated by $S(X)$: \[\Inn R:=\langle s_x\mid x\in X\rangle.\]
\end{definition}

\begin{obs}
    For all racks $R$, the group $\Inn R$ is a normal subgroup of $\Aut R$.
\end{obs}

\begin{example}[\cite{fenn}*{Sec.\ 1.3}]\label{ex:innconj}
    Let $G$ be a group, let $\Conj X$ be a subquandle of $\Conj G$, and let $H:=\langle X\rangle$ be the subgroup of $G$ generated by $X$. Then $\Inn(\Conj X)$ equals the inner automorphism group $\Inn H$ of the group $H$. In particular, the assignment $s_x\mapsto xZ(H)$ is a group isomorphism from $\Inn(\Conj X)$ to $H/Z(H)$.
\end{example}

\begin{example}
    The inner automorphism group of the tetrahedral quandle from Example \ref{ex:tetra} is the alternating group $A_4$.
\end{example}

\begin{definition}
    We say that a rack $R$ is \emph{connected} if the action of $\Inn R$ on $R$ is transitive.
\end{definition}

\begin{lemma}[\cite{conjugation}*{Thm.\ 3.2}]\label{lem:conn}
    If $\Conj X$ is a subquandle of $\Conj G$, then $\Conj X$ is connected if and only if $X$ is a conjugacy class in the subgroup $\langle X\rangle\leq G$.
\end{lemma}

\begin{lemma}[\cite{core2}*{Cor.\ 3.7}]\label{lem:core-conn}
    Let $G$ be a group. Then the core quandle $\Core G$ is connected if and only if $G$ is generated by the squares of its elements.
\end{lemma}

\section{Preliminaries: Good involutions}\label{sec:prelims2}

\subsection{Rack antiautomorphisms} A key ingredient in our proofs is \emph{rack antiautomorphisms}, which Horvat \cite{antiaut} introduced in 2017 to study homeomorphisms of pairs $(S^3,K)$ for all knots $K\subset S^3$.

\begin{definition}
    Given a rack $R=(X,s)$, define $s\inv:X\to S_X$ by $x\mapsto s_x\inv$. Then the rack $R\op:=(X,s\inv)$ is a rack called the \emph{dual rack} of $R$.
\end{definition}

\begin{definition}
    Given a rack $R$, an \emph{antiautomorphism} of $R$ is a rack isomorphism $\rho:R\bij R\op$. Let $\Anti R$ be the set of antiautomorphisms of $R$. We say that $R$ is \emph{self-dual} if $\Anti R$ is nonempty. 
\end{definition}

Concretely, $\Anti R$ is the set of permutations $\rho\in S_X$ that satisfy
\[
    \rho s_x=s\inv_{\rho(x)}\rho
\]
for all $x\in X$.

\begin{obs}[\cite{antiaut}*{Rem.\ 2.6}]\label{obs:anticomp}
    The composition of two rack antiautomorphisms is a rack automorphism, while the composition of an antiautomorphism and an automorphism (in either order) is an antiautomorphism.
\end{obs}

\begin{example}\label{ex:inv-sd}
    Let $R=(X,s)$ be a rack. Then the following are equivalent:
    \begin{itemize}
        \item $R$ is involutory.
        \item $R$ is equal (not only isomorphic) to the dual rack $R\op$.
        \item $\id_X\in \Anti R$.
    \end{itemize}
    Under the above conditions, $R$ is self-dual, and $\Anti R=\Aut R$.
\end{example}

\begin{example}\label{ex:subconj}
Let $G$ be a group, let $\iota\in S_G$ be the inversion map $g\mapsto g\inv$, and let $\Conj X$ be an inversion-closed subquandle of $\Conj G$. Then the restriction $\iota|_X$ is an involution of $X$. Moreover, $\iota|_X$ is an antiautomorphism of $\Conj X$ because
    \[
    \iota s_y(x)=(yxy\inv)\inv = yx\inv y\inv=\iota(y)\inv \iota(x)\iota(y)=s\inv_{\iota(y)} \iota(x)
    \]
for all $x,y\in X$. In particular, $\Conj X$ is self-dual. Note that $\iota\in\Aut(\Conj X)$ if and only if $\Conj X$ is a kei.
\end{example}

\begin{example}[\cite{tetrahedral}*{Ex.\ 2.16}]\label{ex:tetra-sd}
    The tetrahedral quandle $Q$ from Example \ref{ex:tetra} is self-dual; indeed, $\Anti Q$ contains every transposition in $S_4$.
\end{example}

\begin{example}
    Not all quandles are self-dual.
    For example, let $n\in\Z^+$ and $a\in\Z$ be relatively prime integers, define the quotient ring $X:=(\Z/n\Z)[t^{\pm 1}]/(t-a)$, and define $\phi:X\to X$ by $x\mapsto (1-t)x$. Then $\Alex(X,\phi)$ is called a \emph{linear Alexander quandle}. A result of Nelson \cite{alex}*{Cor.\ 2.8} states that $\Alex(X,\phi)$ is self-dual if and only if $a$ is a square modulo $n/{\gcd(n,1-a)}$.
\end{example}

\subsection{Good involutions} This paper concerns the category of \emph{symmetric racks}, which Kamada \cite{symm-quandles} introduced in 2007.

\begin{definition}\label{def:symmetric}
    Let $R=(X,s)$ be a rack. A \emph{good involution} of $R$ is an involution $\rho\in S_X$ that satisfies the equalities
    \[
    \rho s_x=s_x\rho,\qquad s_{\rho(x)}=s\inv_x
    \]
    for all $x\in X$. We denote the set of good involutions of $R$ by $\Good R$.

    A \emph{symmetric rack} (resp.\ \emph{symmetric quandle}) is a pair $(R,\rho)$ where $R$ is a rack (resp.\ quandle) and $\rho\in \Good R$.
\end{definition}

\begin{obs}\label{obs:iso-good}
    The set of good involutions is a rack invariant.
    Explicitly, if $\phi:R_1\bij R_2$ is a rack isomorphism, then the assignment $\rho\mapsto \phi \rho\phi\inv$ is a bijection from $\Good R_1$ to $\Good R_2$.
\end{obs}

Symmetric racks form a category with the following morphisms; symmetric quandles form a full subcategory of this category.

\begin{definition}
    A \emph{symmetric rack homomorphism} between two symmetric racks $(R_1,\rho_1)$ and $(R_2,\rho_2)$ is a rack homomorphism $\varphi$ from $R_1$ to $R_2$ that intertwines the good involutions: \[\varphi\rho_1=\rho_2\varphi.\]
    A \emph{symmetric rack isomorphism} is a bijective symmetric rack homomorphism.
\end{definition}

\begin{example}[\cite{symm-quandles-2}*{Prop.\ 3.1}]\label{ex:trivial}
    Given a quandle $R=(X,s)$, let $I\subset S_X$ the set of all involutions of $X$. Then $\Good R=I$ if and only if $R$ is a trivial quandle.
\end{example}

\begin{example}[\cite{symm-quandles}*{Ex.\ 2.4}]\label{ex:conjinv}
    Let $G$ be a group. Then the inversion map $\iota\in S_G$ defined by $g\mapsto g\inv$ is a good involution of $\Conj G$.
\end{example}

\begin{example}\label{ex:prelim-conj}
    Example \ref{ex:conjinv} immediately implies its own generalization. 
    Let $G$ be a group, let $\iota\in S_G$ be the inversion map $g\mapsto g\inv$, and let $\Conj X$ be a subquandle of $\Conj G$. Then the restriction $\iota|_X$ is a good involution of $\Conj X$ if and only if $\Conj X$ is inversion-closed; cf.\ Section~\ref{sec:conj}.
\end{example}

\begin{example}[\cite{karmakar}*{Ex.\ 2.3}]\label{ex:core-gi}
    Let $G$ be a group, and let $z\in Z(G)$ be a central element of order $1$ or $2$. Then multiplication by $z$ is a good involution of $\Core G$.
\end{example}

\begin{example}
    Let $R=\Alex(G,\phi)$ be a generalized Alexander quandle. A theorem of Taniguchi \cite{taniguchi}*{Thm.\ 1.3} states that $\Good R$ is nonempty if and only if $R$ is a kei.
\end{example}

\subsubsection{Involutory symmetric racks}

We recall a well-known result of Kamada and Oshiro \cite{symm-quandles-2}. 

\begin{prop}[\cite{symm-quandles-2}*{Prop.\ 3.4}]\label{ex:inv-gi}
    Let $R=(X,s)$ be a rack. Then the following statements are equivalent:
    \begin{itemize}
        \item $R$ is involutory.
        \item $\id_X\in \Good R$.
        \item $\Good R$ contains an element of $\Aut R$.
        \item $\Good R$ is a nonempty subset of $\Aut R$.
    \end{itemize}
\end{prop}

\begin{rmk}
    Although Kamada and Oshiro only stated their result for quandles, their proof holds more generally for racks.
\end{rmk}

In fact, we have the following.

\begin{prop}\label{cor:inv-cent}
    If $R$ is involutory, then $\Good R$ is precisely the set $I$ of involutions $\rho:X\to X$ contained in the centralizer $C_{\Aut R}(\Inn R)$.
\end{prop}

\begin{proof}
    The containment $\Good R\subseteq I$ follows from Proposition \ref{ex:inv-gi}. Conversely, if $\rho\in I$, then
    \[
    s_{\rho(x)}=\rho s_{\rho(x)}\rho=\rho^2 s_x=s_x=s_x\inv,
    \]
    so $\rho$ is a good involution of $R$.
\end{proof}

\begin{cor}
    Suppose that $R$ is involutory. If $\Aut R$ is a simple group, then $\Good R$ is either $\{\id_X\}$ or the set of all involutory automorphisms of $R$.
\end{cor}

\begin{proof}
    Since $\Inn R$ is a normal subgroup of $\Aut R$, so is the centralizer $C_{\Aut R}(\Inn R)$. Hence, the claim follows from Proposition \ref{cor:inv-cent}.
\end{proof}

\section{General results for racks}\label{sec:general}
Let $R=(X,s)$ be a rack. 
In this section, we characterize good involutions $\rho\in\Good R$ as involutory central antiautomorphisms of $R$ (Proposition \ref{prop:int}). As a result, self-duality is a necessary condition for $R$ to have good involutions (Corollary \ref{prop:sd}), and every good involution factors through some automorphism $\phi$ and a given antiautomorphism $\iota$ (Corollary \ref{thm:main}).

\subsection{Good involutions are antiautomorphisms} We record a ``two-implies-three'' rule for good involutions. This will facilitate the proof of Theorem \ref{thm2} later in the paper.

\begin{lemma}\label{prop:pre-int}
    For all permutations $\rho\in S_X$, any pair of the following three conditions implies the~other:
    \begin{enumerate}[\textup{(A\arabic*)}]
        \item\label{Z1} For all $x\in X$, $s_{\rho(x)}\inv=s_x$.
        \item\label{Z2} $\rho$ lies in the centralizer $C_{S_X}(\Inn R)$.
        \item\label{Z3} $\rho$ is an antiautomorphism of $R$.
    \end{enumerate}
\end{lemma}

\begin{proof}
    Straightforward.
\end{proof}

\begin{prop}\label{prop:int}
    For all involutions $\rho\in S_X$, the following are equivalent:
    \begin{itemize}
        \item $\rho$ is a good involution of $R$.
        \item $\rho$ satisfies at least two out of the three conditions in Lemma \ref{prop:pre-int}.
        \item $\rho$ satisfies all three of the conditions in Lemma \ref{prop:pre-int}.
    \end{itemize}
\end{prop}

\begin{proof}
    Combined with the assumption that $\rho$ is an involution, conditions \ref{Z1} and \ref{Z2} form the definition of a good involution. Hence, the claim follows immediately from Lemma \ref{prop:pre-int}.
\end{proof}

\begin{rmk}
    Together, Example \ref{ex:inv-sd} and Proposition \ref{prop:int} recover Proposition \ref{ex:inv-gi}.
\end{rmk}

The following addresses Problem \ref{prob2}.

\begin{cor}\label{prop:sd}
    If $R$ is not self-dual, then $R$ has no good involutions.
\end{cor}

\begin{proof}
    By definition, $R$ is self-dual if and only if $\Anti R$ is nonempty. 
    By Proposition \ref{prop:int}, $\Good R$ is a subset of $\Anti R$, so the claim follows.
\end{proof}

\begin{rmk}\label{rmk:galkin}
    The inverse of Corollary \ref{prop:sd} is not true in general, even if $\Anti R$ contains involutions. For example, all transpositions in the symmetric group $S_4$ are antiautomorphisms of the tetrahedral quandle $Q$ from Example \ref{ex:tetra}. Nevertheless, Table \ref{table:counts} implies that $Q$ has no good involutions; cf.\ Example \ref{ex:a4-2}. In fact, there are infinitely many examples of this nature; see Proposition \ref{prop:galkin}.
\end{rmk}

\subsection{Factoring good involutions}

In light of Corollary \ref{prop:sd}, assume for the rest of this section that $R$ is self-dual, and fix an antiautomorphism $\iota\in \Anti R$. 

\begin{prop}\label{prop:tildebij}
    The assignment $\phi\mapsto \iota\phi$ is a bijection from $\Aut R$ to $\Anti R$.
\end{prop}

\begin{proof}
    Since $R\cong R\op$, this follows from Observation \ref{obs:anticomp} and the definition of $\Anti R$.
\end{proof}

The following corollaries follow directly from Propositions \ref{prop:int} and \ref{prop:tildebij}. In particular, Corollary \ref{thm:main} states that knowing the elements of $\Aut R$ is sufficient for computing $\Good R$. 

\begin{cor}\label{thm:main}
    For all involutions $\rho\in S_X$, $\rho$ is a good involution of $R$ if and only if there exists an automorphism $\phi\in \Aut R$ such that $\rho=\iota\phi$ and at least one of the conditions \ref{Z1}, \ref{Z2} hold.
\end{cor}

\begin{cor}\label{cor:sharp}
    The number of good involutions of $R$ is at most the order of $\Aut R$:
    \[
    |{\Good R}|\leq |{\Anti R}|= |{\Aut R}|.
    \]
    Equality holds if and only if $\Anti R$ only contains involutions in the centralizer $C_{S_X}(\Inn R)$.
\end{cor}

\begin{rmk}
    The bound in Corollary \ref{cor:sharp} is sharp in the sense that there exist quandles $R$ of infinitely many orders $n$ such that ${\Good R}= {\Anti R}$. To see this, let $n\geq 2$ be an even integer, and define a rack structure $s$ on the set $X=\{1,2,\dots,n\}$ in cycle notation by
    \[
    s_1,s_2:=\id_X;\qquad s_k,s_{k+1}:=(k-2\;\;k-1)
    \]
    for all odd $3\leq k<n$. Then $R:=(X,s)$ is a kei, so Example \ref{ex:inv-sd} implies that $\Anti R$ is the elementary abelian $2$-group
    \[
    \Anti R=\Aut R=\langle (12),(34),\dots,(n-1\;\; n) \rangle \cong (\Z/2\Z)^{n/2}.
    \]
    In particular, $R$ satisfies the last condition in Corollary \ref{cor:sharp}, so ${\Good R}= {\Anti R}$.
\end{rmk}

\section{Faithful symmetric racks}\label{sec:special}
In this section, we study Problems \ref{prob1} and \ref{prob2} for faithful racks. 

\subsection{General results} 

\begin{prop}\label{prop:faithful}
    If $R$ is faithful, then $R$ has at most one good involution. 
\end{prop}

\begin{proof}
    Let $\rho,\rho'\in \Good R$. For all $x\in X$,
    \[
    s_{\rho(x)}=s\inv_x=s_{\rho'(x)},
    \]
    so $\rho(x)=\rho'(x)$ because $R$ is faithful. Therefore, $\rho=\rho'$.
\end{proof}

\begin{rmk}
    Example \ref{ex:tetra} and Proposition \ref{prop:galkin} give examples of faithful, self-dual quandles with no good involutions.
    
    On the other hand, Clark et al.\ \cite{galkin-symm}*{Prop.\ 5.13} provided infinitely many faithful, self-dual quandles that do have a (necessarily unique) good involution. The next two results do the same.
\end{rmk}

\begin{cor}\label{cor:centerless}
    Let $G$ be a group, let $\Conj X$ be an inversion-closed subquandle of $\Conj G$, and let $\langle X\rangle$ be the subgroup of $G$ generated by $X$. If $\langle X\rangle$ is centerless, then $\Good(\Conj X)=\{\iota|_X\}$.
\end{cor}

\begin{proof}    
    By assumption, the action of $\langle X\rangle$ on $X$ by conjugation is faithful, so Observation \ref{obs:conj2} states that $\Conj X$ is faithful. Hence, the claim follows from Example \ref{ex:prelim-conj} and Proposition \ref{prop:faithful}. 
\end{proof}

\begin{rmk}
    An inverse to Corollary \ref{cor:centerless} holds; see Corollary \ref{cor:centerless2}.
\end{rmk}

\begin{cor}\label{cor:faithful-inv}
    If $R$ is faithful and involutory, then $\Good R=\{\id_X\}$.
\end{cor}

\begin{proof}
This follows from Propositions \ref{ex:inv-gi} and \ref{prop:faithful}.
\end{proof}

\begin{example}\label{eexalex-faith}
    Let $G$ be a group, and let $\phi\in\Aut G$ be an automorphism of $G$. By a theorem of Taniguchi \cite{taniguchi}*{Thm.\ 1.3}, the generalized Alexander quandle $\Alex(G,\phi)$ has a good involution if and only if it is a kei. In this case, if the only fixed point of $\phi$ is the identity element $1\in G$, then Example \ref{ex:alex} and Corollary \ref{cor:faithful-inv} imply that $\Good(\Alex(G,\phi))=\{\id_G\}$.

    In particular, if $G$ is abelian and $2$-torsionless, then the only good involution of the Takasaki kei $T(G)$ is the identity map $\id_G$; cf.\ Corollary \ref{cor:taka}. This generalizes a result of Kamada and Oshiro \cite{symm-quandles-2}*{Thm.\ 3.2} for dihedral quandles $R_n$ of odd order.
\end{example}

\subsection{Symmetric Galkin quandles}

Generalizing a construction of Galkin \cite{galkin} in 1988, Clark et al.\ \cite{galkin-symm} introduced a family of quandles $\G(A,c)$ called \emph{Galkin quandles} in 2013. Galkin quandles were studied extensively in \citelist{\cite{galkin-1}\cite{galkin-symm}}. 

In particular, Clark et al.\ \cite{galkin-symm}*{Prop.\ 5.13} constructed good involutions of certain Galkin quandles. We show that no other good involutions of Galkin quandles exist. Aside from addressing Problems \ref{prob1} and \ref{prob2}, this is also motivated by finding infinitely many quandles that have involutory antiautomorphisms and no good involutions; cf.\ Remark \ref{rmk:galkin}.

Throughout this section, let $A$ be an additive abelian group.

\begin{definition}[\cite{galkin-1}]
    Fix an element $c\in A$. Define functions ${\mu:\Z/3\Z\to\Z}$ and $\tau:\Z/3\Z\to A$~by \[\mu(n):=\begin{cases} 2&\text{if } n=0,\\ -1& \text{if } n=1,2,\end{cases} \qquad \tau(n):=\begin{cases} 0& \text{if } n=0,1,\\ c& \text{if } n=2.\end{cases}\] 
    Define a rack structure $s$ on the direct product $\Z/3\Z\times A$ by
    \[
    s_{(y,b)}(x,a):=(2y-x,-a+\mu(x-y)b+\tau(x-y)).
    \]
    Then the pair $\G(A,c):=(\Z/3\Z\times A,s)$ is a quandle called a \emph{Galkin quandle}.
\end{definition}

\begin{rmk}
    A similar definition of Galkin quandles also exists \cite{galkin-symm}*{Def.\ 3.4}; its equivalence to the above definition was proven in \cite{galkin-symm}*{Lem.\ 3.6}.
\end{rmk}

\begin{rmk}[\cite{galkin-symm}*{Lem.\ 5.2, Prop.\ 5.11, Cor.\ 5.12}]\phantom{a}
\begin{itemize}
    \item All Galkin quandles are faithful.
    \item For all Galkin quandles $\G(A,c)$, the involution $(x,a)\mapsto (-x,a)$ is an antiautomorphism of $\G(A,c)$. In particular, all Galkin quandles are self-dual.
    \item $\G(A,c)$ is involutory if and only if $c=0$.
\end{itemize}
\end{rmk}

\subsubsection{Results}
Clark et al.\ \cite{galkin-symm}*{Prop.\ 5.13, cf.\ Cor.\ 5.12} constructed good involutions of certain Galkin quandles. Namely, let $c\in A$ be an element such that $2c=0$, and define
\[
\rho_c(x,a):=(x,a+c)
\]
for all $(x,a)\in \Z/3\Z\times A$. Then $\rho_c$ is a good involution of $\G(A,c)$.

\begin{prop}\label{prop:galkin}
    If $\G(A,c)$ is a Galkin quandle, then
    \[
    \Good(\G(A,c))=\begin{cases}
        \{\rho_c\}&\text{if } 2c=0,\\
        \emptyset&\text{otherwise.}
    \end{cases}
    \]
\end{prop}

\begin{proof}
    If $2c=0$, then $\rho_c\in \Good(\G(A,c))$ by the construction of Clark et al. Conversely, for all good involutions $\rho=(\rho_1,\rho_2)$ of $\G(A,c)$ and for all $(x,a),(y,b)\in \Z/3\Z\times A$, the expression
    \[
    s_{\rho(y,b)}(x,a)=(2\rho_1(y)-x,-a+\mu(x-\rho_1(y))\rho_2(b)+\tau(x-\rho_1(y))
    \]
    equals the expression
    \[
    s_{(y,b)}\inv(x,a)=(2y-x,-a+\mu(y-x)b+\tau(y-x)).
    \]
    In particular, $\rho_1$ must be the identity map $\id_{\Z/3\Z}$. Letting $z:=x-y$, we obtain
    \[
    \mu(z)\rho_2(b)+\tau(z)=\mu(-z)b+\tau(-z).
    \]
    Taking $z$ to be $1$ and $2$ respectively yields
    \[
    -\rho_2(b)=-b+c,\qquad -\rho_2(b)+c=-b,
    \]
    from which we deduce that $2c=0$ and $\rho_2(b)=b+c$ for all $b\in A$. In particular, $\rho=\rho_c$.
\end{proof}

\begin{cor}
    If $A$ is $2$-torsionless, then for all $c\in A$, there are no good involutions of $\G(A,c)$.
\end{cor}

\section{Symmetric conjugation quandles}\label{sec:conj} 

In this section, we prove Theorem \ref{thm2} (Theorem \ref{thm:conj}) and two corollaries. We also consider some special cases along the way (Proposition \ref{cor:conj-none}, Corollary \ref{cor:centerless2}).
Henceforth, let $G$ be a group, let $\Conj X$ be a subquandle of $\Conj G$, and let $H:=\langle X\rangle$ be the subgroup of $G$ generated by $X$.

\subsection{Preliminary results} To prove Theorem \ref{thm2}, we use the following lemmas.

\begin{lemma}\label{cor:conjmain}
    For all involutions $\rho\in S_X$, $\rho$ is a good involution of $\Conj X$ if and only if $\rho\in \Anti(\Conj X)$ and $\rho(x)x\in Z(H)$ for all $x\in X$.
\end{lemma}

\begin{proof}
    ``$\implies$'' If $\rho\in \Good(\Conj X)$, then $\rho\in\Anti(\Conj X)$ by Proposition \ref{prop:int}. We have to show that $\rho(x)xy=y\rho(x)x$ for all $x,y\in X$. Since $\rho$ is a good involution,
    \[
    xyx\inv=s_x(y)=s\inv_{\rho(x)}(y)=\rho(x)\inv y\rho(x),
    \]
    and the desired equality follows.

    ``$\impliedby$'' Suppose that $\rho\in \Anti(\Conj X)$ and $\rho(x)x\in Z(H)$ for all $x\in X$. By Proposition \ref{prop:int}, it suffices to show that $s_{\rho(x)}\inv= s_x$ for all $x\in X$. Since $\rho(x)x\in Z(H)$, the equality
    \[
    \rho(x)\inv Z(H)=xZ(H)
    \]
    holds in the quotient group $H/Z(H)$. It follows from Example \ref{ex:innconj} that
    \[
    s_{\rho(x)\inv}=s_x.
    \]
    By the definition of $\Conj X$, the left-hand side equals $s_{\rho(x)}\inv$. This completes the proof.
\end{proof}

\begin{obs}\label{obs:autconj}
    An antiautomorphism $\rho$ of $\Conj X$ is a precisely a permutation of $X$ satisfying
\[    \rho(yxy\inv)=\rho(y)\inv\rho(x)\rho(y)\] for all $x,y\in X$.
\end{obs}

\begin{lemma}\label{lem:classftn}
    Let $\zeta:X\to Z(H)$ be a function, and suppose that the function $\rho:X\to H$ defined~by \[\rho(x):=\ (\zeta(x)x)\inv\] is a permutation of $X$. Then $\rho\in\Anti(\Conj X)$ if and only if $\zeta(yxy\inv)=\zeta(x)$ for all $x,y\in X$.
\end{lemma}

\begin{proof}
    ``$\implies$'' Suppose that $\rho\in \Anti(\Conj X)$. Since $\zeta(X)\subseteq Z(H)$, Observation \ref{obs:autconj} yields
    \begin{align*}
        yx y\inv\zeta(yxy\inv)&=\rho(yxy\inv)\inv\\
        &=\rho(y)\inv\rho(x)\inv\rho(y)\\
        &=yx y\inv\zeta(y) \zeta(x)\zeta(y)\inv\\
        &=yx y\inv \zeta(x)
    \end{align*}
    for all $x,y\in X$. Hence, $\zeta(yxy\inv)=\zeta(x)$.

    ``$\impliedby$'' The proof is similar.
\end{proof}

\subsection{Intermediate results}
The following results address Problems \ref{prob1} and \ref{prob2}.

\begin{prop}\label{cor:conj-none}
    If there exists an element $x\in X$ such that $yx\notin Z(H)$ for all $y\in X$, then $\Conj X$ has no good involutions.
\end{prop}

\begin{proof}
    For all good involutions $\rho$ of $\Conj X$, Lemma \ref{cor:conjmain} states that $\rho(x)x\in Z(H)$ for all $x\in X$, so the claim follows.
\end{proof}

\begin{cor}\label{cor:centerless2}
    Let $\iota:G\to G$ be the inversion map $g\mapsto g\inv$. 
    If $H$ is centerless, then 
    \[
    \Good(\Conj X)=\begin{cases}
        \{\iota|_X\}& \text{if } \iota(X)=X,\\
        \emptyset & \text{otherwise.}
    \end{cases}
    \]
    In particular, if $G$ is centerless, then $\Good(\Conj G)=\{\iota\}$.
\end{cor}

\begin{proof}
    This follows from Corollary \ref{cor:centerless} and Proposition \ref{cor:conj-none}.
\end{proof}

\begin{example}\label{ex:a4-2}
    Consider the quandle $\Conj X$ from Example \ref{ex:a4}. Although $\Conj X$ is self-dual, Corollary \ref{cor:centerless2} states that $\Conj X$ has no good involutions. This can also be seen from the fact that $\Conj X$ is a noninvolutory quandle of order $4$; Table \ref{table:counts} shows that no such quandles have good involutions.
\end{example}

\subsection{Main results}
The following is Theorem \ref{thm2}. After proving the theorem, we discuss its interpretation in terms of orbits.

\begin{thm}\label{thm:conj}
    For all functions $\rho:X\to X$, $\rho$ is a good involution of $\Conj X$ if and only if there exists a function $\zeta:X\to Z(H)$ such that
    \begin{equation}\label{eq:conj-conds}
        \rho(x)=(\zeta(x)x)\inv,\qquad\zeta(yxy\inv)=\zeta(x)=\zeta\rho(x)
    \end{equation}
    for all $x,y\in X$.
\end{thm}

\begin{proof}
    ``$\implies$'' Given $\rho\in \Good(\Conj X)$, Lemma \ref{cor:conjmain} states that we can define $\zeta:X\to Z(H)$ by
    \[
    \zeta(x):=(\rho(x)x)\inv.
    \]
    This way, $\rho$ satisfies the first equality in \eqref{eq:conj-conds}. Since $\rho\in\Anti(\Conj X)$ (by Proposition \ref{prop:int}), the second equality in \eqref{eq:conj-conds} follows from Lemma \ref{lem:classftn}. For all $x\in X$, the centrality of $\zeta(x)\inv=\rho(x)x$ implies that $x\rho(x)=\rho(x)x$, so
    \[
    \zeta\rho(x)=(\rho^2(x)\rho(x))\inv =(x\rho(x))\inv=(\rho(x)x)\inv=\zeta(x)
    \]
    because $\rho$ is an involution. This verifies the final equality in \eqref{eq:conj-conds}.

    ``$\impliedby$'' If $\rho:X\to X$ and $\zeta:X\to Z(H)$ satisfy \eqref{eq:conj-conds}, then for all $x\in X$, the first and third equalities in \eqref{eq:conj-conds} yield
    \[
        \rho^2(x)=(\zeta\rho(x)\rho(x))\inv=(\zeta(x)\rho(x))\inv=(\zeta(x)\zeta(x)\inv x\inv)\inv=x
    \]
    because $\zeta(x)\in Z(H)$. In other words, $\rho$ is an involution, so by Lemma \ref{cor:conjmain}, it suffices to show that $\rho\in\Anti(\Conj X)$ and $p(x)x\in Z(H)$ for all $x\in X$. But the former condition follows from Lemma \ref{lem:classftn}, and the latter follows from the first equality in \eqref{eq:conj-conds}.
\end{proof}

\begin{example}\label{obs:prelim-conj}
    Taking $\rho$ to be the inversion map $x\mapsto x\inv$ corresponds to letting $\zeta:X\to Z(H)$ be the constant function $x\mapsto 1$. Hence, Theorem \ref{thm:conj} recovers Example \ref{ex:prelim-conj}.
\end{example}

\begin{rmk}\label{obs:cf}
    The condition that $\zeta(yxy\inv)=\zeta(x)$ means that $\zeta:X\to Z(H)$ is constant on orbits of $X$ under the action of $H$ by conjugation. These orbits are precisely the connected components of $\Conj X$; cf.\ Lemma \ref{lem:conn}. Denote the set of all such orbits by $X/H$.
    
    Therefore, Theorem \ref{thm:conj} states that every good involution $\rho$ of $\Conj X$ factors through the natural projection $\pi:X\surj X/H$ and a unique function $\zeta^*:X/H\to Z(H)$. Moreover, $\zeta^*$ completely determines $\rho$.
\end{rmk}

\begin{rmk}
    If $X=G$, then the condition that $\zeta(yxy\inv)=\zeta(x)$ states that $\zeta:G\to Z(G)$ is a class function.
\end{rmk}

\subsubsection{Corollaries}
We provide upper bounds on the number of good involutions of $\Conj X$. In the case that $\Conj X$ is connected, we also simplify the characterization of $\Good(\Conj X)$ in Theorem \ref{thm:conj}. 
In the following, let $X/H$ denote the set of orbits of $X$ under the action of $H$ by conjugation.

\begin{cor}\label{cor:connected}
    Suppose that $\Conj X$ is connected. 
    Then for all functions $\rho:X\to X$, $\rho$ is a good involution of $\Conj X$ if and only if there exists a central element $z\in Z(H)$ such that \[\rho(x)=zx\inv\] for all $x\in X$. In particular,
    \[
    |{\Good(\Conj X)}|\leq |{Z(H)}|.
    \]
    If $X$ is not closed under inverses, then this bound is strict.
\end{cor}

\begin{proof}
    By Remark \ref{obs:cf}, every good involution $\rho$ of $\Conj X$ factors through the natural projection $\pi:X\surj X/H$ and a unique function $\zeta^*:X/H\to Z(H)$. 
    By hypothesis, $X/H=\{X\}$ is a singleton set, so $\zeta^*$ is a constant function. Hence, taking \[z:=(\zeta^*(X))\inv\] yields the main claim due to Theorem \ref{thm:conj}. The final claim follows from Example \ref{obs:prelim-conj}.
\end{proof}

In other words, if $\Conj X$ is connected, then the set $\Good(\Conj X)$ is in bijection with the set of central elements $z\in Z(H)$ such that $zx\inv\in X$ for all $x\in X$.

Even if $\Conj X$ is not connected, the following bound holds.

\begin{cor}\label{cor:orbits}
    Let $k(X):=|{X/H}|$. Then
    \[
    |{\Good(\Conj X)}|\leq \min(|{\Aut(\Conj X)}|, |{Z(H)}|^{k(X)}).
    \]
    If $X$ is not closed under inverses, then the bound $|{\Good(\Conj X)}|<  |{Z(H)}|^{k(X)}$ is strict.
\end{cor}

\begin{proof}
    The bound $|{\Good(\Conj X)}|\leq |{\Aut(\Conj X)}|$ comes from Corollary \ref{cor:sharp}. 
    On the other hand, Remark \ref{obs:cf} shows that there are at least as many functions $\zeta^*:X/H\to Z(H)$ as there are good involutions $\rho$; that is, $|{Z(H)}|^{k(X)}\geq |{\Good(\Conj X)}|$. The final claim follows from Example~\ref{obs:prelim-conj}.
\end{proof}

\begin{rmk}\label{rmk:sharp}
    By Corollary \ref{cor:centerless2}, the bounds in Corollaries \ref{cor:connected} and \ref{cor:orbits} are equalities whenever $H$ is centerless and $\Conj X$ is inversion-closed. Furthermore, Example \ref{ex:sl25} provides a connected conjugation subquandle $\Conj X$ of order $30$ that attains these bounds even though $Z(H)$ is nontrivial, and Example \ref{ex:su2} provides a similar example whose order is infinite.
\end{rmk}

\section{Symmetric core quandles}\label{sec:core}
In this section, we prove Theorem \ref{thm1} (Theorem \ref{thm:core}) and discuss several related results. This addresses Problems \ref{prob1} and \ref{prob2} for nontrivial core quandles $\Core G$. 

Given a group $G$, recall that $\Core G$ is trivial if and only if $G$ is an elementary abelian $2$-group. Since good involutions of trivial quandles are already classified (see Example \ref{ex:trivial}), we will let $G$ be a group of exponent greater than $2$ throughout this section.

\subsection{Setup for the proof}\label{subsec:A} 
To prove Theorem \ref{thm1}, we apply Observation \ref{obs:iso-good} and Theorem \ref{thm:conj} to a certain subquandle $\Conj X$ of a conjugation quandle $\Conj K$ such that $\Conj X\cong \Core G$. To study $\Conj X$, we first consider two subgroups of $Z(G)$.

\subsubsection{The subgroups $A$ and $T$}

Throughout this section, let $A$ be the subset of $Z(G)$ consisting of elements $a\in Z(G)$ satisfying the following condition: There exist a positive integer $n\in\Z^+$ and elements $g_1,\ldots,g_{2n}\in G$ such that
\[
a=g_1g_2\inv g_3g_4\inv\cdots g_{2n-1}g_{2n}\inv = g_1\inv g_2 g_3\inv g_4\cdots g_{2n-1}\inv g_{2n}.
\]

\begin{obs}
    $A$ is a subgroup of $Z(G)$ and, hence, a normal subgroup of $G$.
\end{obs}

Throughout this section, let $T$ be the subgroup of $Z(G)$ whose elements have order $1$ or $2$:    \[
    T:=\{z\in Z(G)\mid z^2=1\}.
    \]

\begin{prop}\label{prop:torsion}
    $T$ is a subgroup of $A$, and $T=A$ if $G$ is abelian or $Z(G)=T$.
\end{prop}

\begin{proof}
    Straightforward.
\end{proof}

As an aside, we pose the following question.

\begin{prob}
    Does there exist a group $G$ such that $T\neq A$?
\end{prob}

\subsubsection{Core quandles as conjugation subquandles}

Let $K$ be the semidirect product
\[
K:=(G\times G)\rtimes \Z/2\Z,
\]
where $\Z/2\Z$ acts on $G\times G$ by swapping coordinates. In 2021, Bergman \cite{core}*{(6.5)} showed that the function $\phi:G\to K$ defined by
\begin{equation}\label{eq:phi}
    \phi(g):=(g,g\inv,1)
\end{equation}
is a quandle embedding of $\Core G$ into $\Conj K$. Henceforth, let $X:=\phi(G)$ denote the image of $\phi$, and let $H:=\langle X\rangle$ be the subgroup of $K$ generated by $X$. To prove Theorem \ref{thm1}, it suffices to apply Theorem \ref{thm:conj} to $\Conj X$ by Observation \ref{obs:iso-good}. 

\begin{obs}\label{obs:XH}
    Every element of $X$ has order $2$ in $H$.
\end{obs}

\begin{obs}\label{obs:XY}
    For all $g\in G$, $H$ contains the element $(g,g\inv,0)$.
\end{obs}

\subsubsection{Calculation of the center} To apply Theorem \ref{thm:conj} to $\Conj X$, we first compute $Z(H)$.

\begin{lemma}\label{lem:elem2}
    Let $G$ be a group containing an element $z$ such that $zgz\inv=g\inv$ for all $g\in G$. Then $G$ is an elementary abelian $2$-group.
\end{lemma}

\begin{proof}
    For all $g,h\in G$, conjugating the product $g\inv h\inv$ by $z$ shows that $g$ and $h$ commute. Thus, $G$ is abelian, so
    \[
    g=zgz\inv=g\inv
    \]
    for all $g\in G$.
\end{proof}

\begin{lemma}\label{obs:XZ}
    Let $a\in Z(G)$. Then $(a,a,0)\in H$ if and only if $a\in A$.
\end{lemma}

\begin{proof}
    We show that $a\in A$ if $(a,a,0)\in H$; the proof of the reverse implication is similar.
    Since $(a,a,0)\in H$, Observation \ref{obs:XH} shows that $(a,a,0)$ is a product of an even number of elements of $X$. That is, there exist a positive integer $n\in\Z^+$ and elements $g_1,\ldots,g_{2n}\in G$ such that
    \[
    (a,a,0)=\prod^{2n}_{i=1} (g_i,g_i\inv,1)=\prod^{2n-1}_{i\text{ odd}}(g_ig_{i+1}\inv,g_i\inv g_{i+1},0)=(g_1g_2\inv\cdots g_{2n-1} g_{2n}\inv,g_1\inv g_2 \cdots g_{2n-1}\inv g_{2n},0)
    \]
    in $H$. 
    Since $a\in Z(G)$, the equality of the first two coordinates shows that $a\in A$.
\end{proof}

\begin{prop}\label{lem:core}
    The center of $H$ is
    \[ 
    Z(H)=\{(a,a,0)\mid a\in A\}\cong A.\]
\end{prop}

\begin{proof}
    If $a\in A$, then $(a,a,0)\in H$ by Lemma \ref{obs:XZ}, and verifying that $(a,a,0)\in Z(H)$ is straightforward. 
    
    Conversely, let $(a,b,i)\in Z(H)$ be arbitrary. 
    If $i=1$, then Observation \ref{obs:XY} shows that
    \[
    (g,g\inv,0)=(a,b,1)(g,g\inv,0)(a,b,1)\inv=(ag\inv a\inv,bgb\inv,0)
    \]
    for all $g\in G$. 
    It follows from Lemma \ref{lem:elem2} that $G$ is an elementary abelian $2$-group, which contradicts the assumption we made at the beginning of this section.

    Thus, we must have $i=0$. By Lemma \ref{obs:XZ}, it now suffices to show that $a\in Z(G)$ and $a=b$. Indeed,
    \[
    (g,g\inv,0)=(a,b,0)(g,g\inv,0)(a,b,0)\inv=(aga\inv,b g\inv b\inv,0)
    \]
    for all $g\in G$, so $a\in Z(G)$. On the other hand,
    \[
    (1,1,1)=(a,b,0)(1,1,1)(a,b,0)\inv=(ab\inv ,b a\inv,1),
    \]
    so $a=b$. 
\end{proof}

\subsubsection{Preliminary versions of Theorem \ref{thm1}}
We state two preliminary versions of Theorem \ref{thm1}.
Denote the rack structure of $\Conj X$ by $t$.
\begin{prop}\label{prop:core-conj}
    For all functions $\rho':X\to X$, $\rho'$ is a good involution of $\Conj X$ if and only if there exists a function $\zeta :X\to Z(H)$ such that
    \[
    \rho'(g,g\inv,1)=\zeta (g,g\inv,1)(g,g\inv,1),\qquad \zeta  t_{(g,g\inv,1)}=\zeta =\zeta \rho'
    \]
    for all $(g,g\inv ,1)\in X$. 
\end{prop}

\begin{proof}
    This follows from combining Theorem \ref{thm:conj} with Observation \ref{obs:XH}.
\end{proof}

In the following, denote the rack structure of $\Core G$ by $s$. 

\begin{prop}\label{prop:alpha}
    For all functions $\rho:G\to G$, $\rho$ is a good involution of $\Core G$ if and only if there exists a function $\alpha:G\to A$ such that
    \[
    \rho(g)=\alpha(g)g,\qquad \alpha s_g=\alpha=\alpha\rho
    \]
    for all $g\in G$.
\end{prop}

\begin{proof}
    We will show that the claim is equivalent to Proposition \ref{prop:core-conj}. 
    Given any function $\alpha:G\to A$, define a function $\F(\alpha):X\to H$ by \[\F(\alpha)(g,g\inv,1):=(\alpha(g),\alpha(g),0).\]
    By Proposition \ref{lem:core}, the image of $\F(\alpha)$ lies in $Z(H)$. Conversely, given any function $\zeta :X\to Z(H)$, define a function $\F\inv(\zeta ):G\to G$ by sending $g\in G$ to the first coordinate of $\zeta (g,g\inv,1)\in Z(H)$. 
    By Proposition \ref{lem:core}, the image of $\F\inv(\zeta )$ lies in $A$. Hence, it is straightforward to verify that the assignments $\F$ and $\F\inv$ are mutually inverse.

    Now, recall that the map $\phi$ from \eqref{eq:phi} is a quandle isomorphism from $\Core G$ to $\Conj X$. The equivalence of the claim with Proposition \ref{prop:core-conj} therefore follows from Observation \ref{obs:iso-good} and the following straightforward calculations: For all elements $g\in G$ and functions $\rho:G\to G$ and $\rho':X\to X$ respectively satisfying the conditions of the claim and Proposition \ref{prop:core-conj}, we have
    \[
    \varphi\rho\varphi\inv(g,g\inv,1)=\F(\alpha)(g,g\inv,1)(g,g\inv,1),\qquad \varphi\inv\rho'\varphi(g)=\F\inv(\zeta )(g)g,
    \]
    so that
    \[
    \F(\alpha s_g)=\F(\alpha)t_{(g,g\inv,1)} ,\qquad \F(\alpha\rho)=\F(\alpha)\phi\rho\phi\inv
    \]
    and, similarly,
    \[
    \F\inv(\zeta t_{(g,g\inv,1)})=\F\inv(\zeta )s_g,\qquad \F\inv(\zeta \rho')=\F\inv(\zeta )\phi\inv\rho'\phi.
    \]
\end{proof}

\subsection{Main results}
We now prove Theorem \ref{thm1}. 
Preserve the notation from the previous subsection, so that $T$ denotes the subgroup \(T=\{z\in Z(G)\mid z^2=1\}\).

\begin{thm}\label{thm:core}
    Proposition \ref{prop:alpha} holds when $A$ is replaced with $T$.
\end{thm}

\begin{proof}
    If $\alpha:G\to T$ is a function satisfying the conditions of Proposition \ref{prop:alpha}, then $\alpha$ induces a good involution $\rho$ by Propositions \ref{prop:torsion} and \ref{prop:alpha}. 
    
    Conversely, let $\rho$ be a good involution of $\Core G$. Then $\rho$ is induced by a unique function ${\alpha:G\to A}$ satisfying the conditions of Proposition \ref{prop:alpha}, and for all $g\in G$, we have \[s_{\rho(g)}=s_g\inv=s_g\] because $\Core G$ is involutory. Hence, \cite{core2}*{Lem.\ 4.1} states that
    \[
    T\ni g\inv \rho(g)=g\inv  \alpha(g)g=\alpha(g),
    \]
    where we have used the fact that $\alpha(g)\in Z(G)$. Since $g\in G$ was arbitrary, the image of $\alpha$ lies in $T$, as desired.
\end{proof}

\begin{example}\label{rmk:core-const}
    The constant functions $\alpha:G\to T$ recover Example \ref{ex:core-gi}.
\end{example}

\begin{rmk}\label{rmk:core}
    Let $\oo:=G/\Inn(\Core G)$ be the set of orbits of $G$ under the action of $\Inn(\Core G)$. Then Theorem \ref{thm:core} states that every good involution $\rho$ of $\Core G$ factors through the natural projection $\pi:G\surj \oo$ and a unique function $\alpha^*:\oo\to T$, and $\alpha^*$ completely determines $\rho$. 
\end{rmk}

\subsubsection{Corollaries} Using Remark \ref{rmk:core}, we obtain the following special cases of Theorem \ref{thm:core}.

\begin{cor}\label{ineq:core}
    Let $\oo:=G/\Inn(\Core G)$ be the set of orbits of $G$ under the action of $\Inn(\Core G)$. Then the following inequality holds:
\[
    |T|\leq|{\Good(\Core G)}|\leq |T|^{|\oo|}.
\]
\end{cor}

\begin{proof}
    The lower bound in Corollary \ref{ineq:core} is obtained from the constant functions $\alpha^*:\oo\to T$; cf.\ Example \ref{rmk:core-const}. The upper bound is a consequence of Remark \ref{rmk:core}.
\end{proof}

\begin{rmk}
    In Propositions \ref{prop:dn-sharp}--\ref{prop:zn-sharp}, we show that infinitely many groups $G$ attain the upper bound in Corollary \ref{ineq:core} while exceeding the lower bound. On the other hand, the next two corollaries consider cases in which the lower and upper bounds are equal.
\end{rmk}

\begin{cor}\label{cor:core}
    Suppose that $G$ is generated by the squares of its elements; that is, $G=\langle g^2\mid g\in G \rangle$. Then good involutions $\rho$ of $\Core G$ are precisely the multiplication maps $g\mapsto tg$ with $t\in T$. In particular, $\Good(\Core G)$ is in bijection with $T$.
\end{cor}

\begin{proof}
    By Lemma \ref{lem:core-conn}, the hypothesis is equivalent to the statement that $\Core G$ is connected. That is, the orbit decomposition $\oo$ in Corollary \ref{ineq:core} is a singleton set, so the claim follows.
\end{proof}

\begin{cor}\label{cor:taka}
    If $Z(G)$ is $2$-torsionless, then the only good involution of $\Core G$ is the identity map $\id_G$.
\end{cor}

\begin{proof}
    This follows immediately from Theorem \ref{thm:core}.
\end{proof}

\begin{rmk}
    Corollary \ref{cor:taka} can also be obtained from Example \ref{ex:core} and Corollary \ref{cor:faithful-inv}.
\end{rmk}

\begin{rmk}
    If $G$ is abelian, then Theorem \ref{thm:core} and Corollary \ref{cor:core} specialize to characterizations of good involutions of Takasaki kei $T(G)$. In particular, Corollary \ref{cor:core} applies to all Takasaki kei of $2$-divisible abelian groups; cf.\ \cite{flat}. Moreover, Corollary \ref{cor:taka} vastly generalizes a result of Kamada and Oshiro \cite{symm-quandles-2}*{Thm.\ 3.2} for dihedral quandles of odd order; cf.\ Example \ref{eexalex-faith}.
\end{rmk}

\section{Algorithms}\label{sec:algos}
As applications of Theorems \ref{thm:conj} and \ref{thm:core}, we describe procedures for computing $\Good R$ when $R$ is a conjugation subquandle $\Conj X$ or a core quandle $\Core G$. We also discuss computational findings for groups up to order $23$.

Carrying over our notation from Section \ref{sec:conj}, let $G$ be a group, let $\Conj X$ be a subquandle of $\Conj G$, and let $H:=\langle X\rangle\leq G$. 

\subsection{How to compute $\Good(\Conj X)$} 
Iterating the following procedure constructs all good involutions $\rho$ of $\Conj X$. The procedure is based on Theorem \ref{thm:conj}; cf.\ Remark \ref{obs:cf}.
\begin{alg}\label{alg}
    Let $X/H$ be the set of orbits of $X$ under the action of $H$ by conjugation (that is, the canonical action of $\Inn X$), and let $\pi:X\surj X/H$ be the natural projection.
    \begin{enumerate}[\textup{(\arabic*)}]
    \item First, pick any function $\zeta^*:X/H\to Z(H)$, and define functions $\zeta:X\to Z(H)$ and $\rho:X\to H$ by
    \[
    \zeta:=\zeta^*\pi,\qquad \rho(x):=(\zeta(x)x)^{-1}.
    \]
    \item If $\rho(X)\subseteq X$ and $\zeta\rho=\zeta$, then Theorem \ref{thm2} states that $\rho$ is a good involution, so return $\rho$. Else, return to the previous step and pick a different function $\zeta^*$.
\end{enumerate}
To compute $\Good(\Conj X)$, repeat these steps until all possible functions $\zeta^*$ have been considered.
\end{alg}
Theorem \ref{thm:conj} states that $\Good(\Conj X)$ is precisely the set of involutions $\rho$ that can be produced by the above method; cf.\ Remark \ref{obs:cf}. 

\begin{rmk}
    If $\Conj X$ is connected, then Algorithm \ref{alg} simplifies via Corollary \ref{cor:connected} in the obvious way: For each central element $z\in Z(H)$, test whether $zx\inv\in X$ for all $x\in X$. The good involutions $\rho$ of $\Conj X$ are precisely the maps $x\mapsto zx\inv$ for which this test succeeds.
\end{rmk}

\subsubsection{Example of Algorithm \ref{alg}}
In the following example, we use Algorithm \ref{alg} to produce a good involutions of $\rho$ of a conjugation quandle such that $\rho$ is neither the inversion map nor the identity~map. 

\begin{example}\label{ex:dihedral}
    Let $n\geq 4$ be an even integer, and let $D_n$ denote the dihedral group of order $2n$:
\[
D_n=\langle r,s\mid r^n=1=s^2,\ srs=r\inv\rangle.
\] Following Algorithm \ref{alg}, we construct a good involution $\rho$ of $\Conj D_n$. 

    Since $n$ is even, the center of $D_n$ is $Z(D_n)=\{1,r^{n/2}\}$. Define the function $\zeta^*$ on conjugacy classes $C$ of $D_n$ by
    \[
    \zeta^*(C):=\begin{cases}
        1&  \text{if } C\subset  Z(D_n),\\
        r^{n/2}& \text{otherwise,}    \end{cases}
    \]
    so that $\zeta:=\zeta^*\pi$ is defined on elements $g$ of $D_n$ by
    \[
    \zeta(g)=\begin{cases}
        1& \text{if } g\in Z(D_n),\\
        r^{n/2}& \text{otherwise.}
    \end{cases}
    \]
    Then $\rho: D_n\to D_n$ is the function defined by
    \[
    \rho(g):=(\zeta(g)g\inv)\inv=\begin{cases}
        g &\text{if } g\in Z(D_n),\\
        r^{n/2-k}& \text{if } g=r^k \text{ where }1\leq k<n \text{ and } k\neq\frac{n}{2},\\
        sr^{n/2+k} &\text{if } g=sr^k \text{ for some }0\leq k<n.
    \end{cases}
    \]
    Verifying that $\zeta\rho=\zeta$ is straightforward.
    Hence, Theorem \ref{thm:conj} states that $\rho$ is a good involution of $\Conj D_n$. Note that $\rho$ does not equal the inversion map $g\mapsto g\inv$.
\end{example}

Combined with Corollary \ref{cor:centerless2} (or Corollary \ref{cor:centerless}), Example \ref{ex:dihedral} yields the following.

\begin{prop}
    Let $n\geq 3$ be an integer. Then the conjugation quandle $\Conj D_n$ has more than one good involution if and only if $n$ is even.
\end{prop}

\subsection{How to compute $\Good(\Core G)$} 

Let $G$ be a group of exponent greater than $2$, so that $\Core G$ is a nontrivial quandle.  
Similarly to the procedure for $\Conj X$, the following procedure constructs all good involutions $\rho$ of $\Core G$. The procedure is based on Theorem \ref{thm:core}; cf.\ Remark~\ref{rmk:core}.

\begin{alg}\label{alg2}
    Let $\oo$ be the set of orbits of $G$ under the action of $\Inn(\Core G)$, and let $\pi:G\surj \oo$ be the natural projection. Also, let \(T:=\{z\in Z(G)\mid z^2=1\}\).
    \begin{enumerate}[\textup{(\arabic*)}]
    \item First, pick any function $\alpha^*:\oo\to T$, and define functions $\alpha:G\to T$ and $\rho:G\to G$ by
    \[
    \alpha:=\alpha^*\pi,\qquad \rho(g):=\alpha(g)g.
    \]
    \item If $\alpha\rho=\alpha$, then Theorem \ref{thm1} states that $\rho$ is a good involution of $\Core G$, so return $\rho$. Else, return to the previous step and pick a different function $\alpha^*$.
\end{enumerate}
To compute $\Good(\Core G)$, repeat these steps until all possible functions $\alpha^*$ have been considered.
\end{alg}
Theorem \ref{thm:core} states that $\Good(\Core G)$ is precisely the set of involutions $\rho$ that can be produced by the above method; cf.\ Remark \ref{rmk:core}.

\subsubsection{Examples of Algorithm \ref{alg2}}
The following three examples show that the upper bound in Corollary \ref{ineq:core} is sharp for infinitely many nontrivial, nonconnected core quandles.

Given an integer $n\geq 2$, let $\Dic_n$ denote the dicyclic group of order $4n$:
\[\Dic_n=\langle a,x\mid a^{2n}=1,\ x^2=a^n,\ xax\inv=a\inv\rangle.\]

\begin{prop}\label{prop:dn-sharp}
    For all odd integers $n\geq 3$, the core quandle $\Core \Dic_n$ has exactly $2^{(n+3)/2}$ good involutions. In particular, $\Dic_n$ exceeds the lower bound and attains the upper bound in Corollary~\ref{ineq:core}.
\end{prop}

\begin{proof}
    As in Algorithm \ref{alg2}, let $\pi:\Dic_n\surj\oo$ be the natural projection, and note that $T=\{1,x^2\}$. On the other hand, routine calculations show that $\oo$ contains exactly $(n+3)/2$ orbits:
    \begin{itemize}
        \item one orbit $\pi(x)=\{x,x^3\}$ of size $2$,
        \item one orbit $\pi(1)=\{a^m\mid 1\leq m\leq 2n\}$ of size $2n$, and
        \item the following $(n-1)/2$ orbits of size $4$:
        \[
    \pi(xa^m)=\{xa^m,xa^{m+1},xa^{2n-m-1},xa^{2n-m}\}\qquad \text{for all }1\leq m\leq \frac{n-1}{2}.
    \]
    \end{itemize}
    In particular, the lower and upper bounds in Corollary \ref{ineq:core} equal $2$ and $2^{(n+3)/2}$, respectively.
    
    Evidently, each orbit in $\oo$ is closed under multiplication by elements of $T$. That is,
    \[
    \pi(tg)=\pi(g)
    \]
    for all $t\in T$ and $g\in \Dic_n$. In particular, given any function $\alpha^*:\oo\to T$, letting $g\in G$ be arbitrary and taking $t:=\alpha^*\pi(g)$ yields
    \[
    \pi(\alpha^*\pi (g)g)=\pi(g)
    \]
    for all $g\in G$. 
    Applying $\alpha^*$ to both sides recovers the condition that $\alpha\rho=\alpha$ in Algorithm \ref{alg2}. Hence, Theorem \ref{thm:core} states that $\alpha^*$ defines a unique good involution $\rho$ of $\Core \Dic_n$. 
    
    Since $\alpha^*:\oo\to T$ was an arbitrary function, Theorem \ref{thm:core} states that the $2^{(n+3)/2}$ possible functions $\alpha^*$ collectively define $2^{(n+3)/2}$ distinct good involutions $\rho$ of $\Core D_n$, and no other good involutions exist. This completes the proof.
\end{proof}

Similar arguments show the following.

\begin{prop}
    For all even integers $n\geq 2$, the core quandle $\Core \Dic_n$ has exactly $2^{n/2+3}$ good involutions. In particular, $\Dic_n$ exceeds the lower bound and attains the upper bound in Corollary~\ref{ineq:core}.
\end{prop}

\begin{prop}[cf.\ \cite{symm-quandles-2}*{Thm.\ 3.2}]\label{prop:zn-sharp}
    If $n\geq 4$ is a multiple of $4$, then the dihedral quandle $R_n=T(\Z/n\Z)$ has exactly four good involutions. In particular, $\Z/n\Z$ exceeds the lower bound and attains the upper bound in Corollary~\ref{ineq:core}.
\end{prop}

\subsection{Enumeration}
We implemented Algorithms \ref{alg} and \ref{alg2} in \texttt{GAP} \cite{GAP4} and used them to compute $\Good(\Conj G)$ and $\Good(\Core G)$ for various groups $G$ up to order $23$. We provide our code and raw data (including the explicit mappings $\rho:G\to G$) in a GitHub repository \cite{code}.

Table \ref{table:conj} counts good involutions of
conjugation quandles $\Conj G$ of nonabelian groups up to order $22$. Table \ref{table:core} does the same for core quandles $\Core G$ of groups of exponent greater than $2$ up to order $23$. In Appendix \ref{app:b}, we specify these counts for each group $G$ considered in Tables \ref{table:conj} and~\ref{table:core}. 

\begin{table}
\centering
\caption{Number of good involutions across all nontrivial conjugation quandles $\Conj G$ of groups $G$ up to order $22$. Cf.\ Table \ref{tab:app1}.}
\label{table:conj}
\begin{tabular}{l|cccccccccc}
Order of $G$                        & 6 & 8 & 10 & 12 & 14  & 16  & 18   & 20    & 21 & 22   \\ \hline
Good involutions           & 1 & 32 & 1 & 17 & 1 & 13056 & 66 & 33 & 1 & 1 
\end{tabular}
\end{table}

\begin{table}
\centering
\caption{Number of good involutions across all nontrivial core quandles $\Core G$ of groups $G$ up to order $23$. Cf.\ Table \ref{tab:app2}.}
\label{table:core}
\begin{tabular}{l|ccccccccccc}
Order of $G$                        & 3 & 4 & 5 & 6 & 7  & 8  & 9   & 10    & 11 & 12 & 13   \\ \hline
Good involutions           & 1 & 4 & 1 & 3 & 1 & 72 & 2 & 3 & 1 & 31 & 1 \\
\end{tabular}
\begin{tabular}{l|cccccccccc}
Order of $G$                        & 14 & 15 & 16 & 17 & 18  & 19  & 20   & 21    & 22 & 23    \\ \hline
Good involutions           & 3 & 1 & 10856 & 1 & 7 & 1 & 47 & 2 & 3 & 1
\end{tabular}
\end{table}

\section{Constructions}\label{sec:construct}
In this section, we compute $\Good(\Conj X)$ for various connected conjugation subquandles $\Conj X$. This addresses Problem \ref{prob2}. By considering examples in which $\Conj X$ is inversion-closed, we also construct infinite families of CNS-quandles. This addresses Problems \ref{prob:cns} and \ref{prob:cns2}.

Carrying over our notation from Section \ref{sec:conj}, let $G$ be a group, let $\Conj X$ be a subquandle of $\Conj G$, and let $H:=\langle X\rangle\leq G$. Also, let $\iota:G\to G$ denote the inversion map $g\mapsto g\inv$.

\subsection{Constructing CNS-quandles}

The following provides a robust, group-theoretic way to construct CNS-quandles.

\begin{prop}\label{prop:cns}
    Suppose that $\Conj X$ satisfies the following three conditions:
    \begin{itemize}
        \item $X$ is a conjugacy class of $H$.
        \item There exists an element $x\in X$ such that $x^2\notin Z(H)$.
        \item $\Conj X$ is inversion-closed.
    \end{itemize}
    Then $(\Conj X,\iota|_X)$ is a CNS-quandle. 
\end{prop}

\begin{proof}
    This follows from Lemma \ref{lem:conn}, Observation \ref{obs:conj2}, and Example \ref{ex:prelim-conj}.
\end{proof}

\subsection{Examples} In the following examples, we study good involutions of connected subquandles of special linear groups $\SL(m,q)$, symmetric groups $S_n$, and the special unitary group $\SU(2)$.

\begin{example}\label{ex:sl25}
    Let $G$ be the special linear group $\SL(2,5)$. We study good involutions of several connected subquandles of $\Conj G$. Using \texttt{GAP} \cite{GAP4}, we find that $G$ has seven conjugacy classes $X$ of size greater than $1$: four of size $12$, two of size $20$, and one of size $30$. 
    
    It is well-known that the only nontrivial normal subgroup of $G$ is the center \[Z(G)=\{\pm I_2\}.\] It follows that each of the seven aforementioned conjugacy classes $X$ generates $G$, so by Lemma \ref{lem:conn}, all of them are connected subquandles $\Conj X$ of $\Conj G$. Using \texttt{GAP}, we verify that the subquandles of orders $12$ and $20$ satisfy the conditions in Proposition \ref{prop:cns}. Therefore, $\Conj G$ contains four CNS-subquandles $(\Conj X,\iota|_X)$ of order $12$ and two CNS-subquandles of order $20$.

    On the other hand, the subquandle $\Conj X$ of order $30$ satisfies only the first two conditions in Proposition \ref{prop:cns}, so $(\Conj X,\iota|_X)$ is a connected symmetric kei. Since the element $-I_2\in Z(G)$ corresponds to the identity map in Corollary \ref{cor:connected}, it follows from Proposition \ref{ex:inv-gi} and Corollary \ref{cor:connected} that
    \[
    \Good(\Conj X)=\begin{cases}
        \{\iota|_X, \id_X\} & \text{if } |X|=30,\\ \{\iota|_X\} & \text{if }|X|=12,20.
    \end{cases}
    \]
\end{example}

\begin{example}\label{ex:sl}
    Let $m\geq 3$ be an integer, and let $\mathbb{F}_q$ be a finite field of order $q$ and characteristic $p\geq 3$. Let $G$ be the special linear group $\SL(m,q)$, and let $X\subset G$ be the conjugacy class of the~matrix
    \[
    M:=\begin{pmatrix}
        1 & 1 &\\
        & 1 &\\
        & & I_{n-2}
    \end{pmatrix}.
    \]
    Since $m\geq 3$, $X$ is the set of all transvections (that is, nonidentity linear transformations of $\mathbb{F}_q^m$ that fix a hyperplane) in $G$, which is a symmetric generating subset of $G$; see, for example, \cite{suzuki}*{Ch.\ 1, Sec.\ 9}. In particular,
    \[
    Z(H)=Z(G)=\{\lambda I_n\mid \lambda^n=1\},
    \]
    so we have
    \[
    M^2= \begin{pmatrix}
        1 & 2 &\\
        & 1 &\\
        & & I_{n-2}
    \end{pmatrix}\notin Z(H)
    \]
    because $p>2$. Hence, Proposition \ref{prop:cns} states that $(\Conj X,\iota|_X)$ is a CNS-quandle. In particular, the formula for $|X|$ given in \cite{fulman}*{Sec.\ 2} yields the first part of Proposition \ref{prop:sl}.

    Finally, we show that $\iota|_X$ is the \emph{only} good involution of $\Conj X$. By Corollary \ref{cor:connected} and Example \ref{obs:prelim-conj}, it suffices to show that there exists a transvection $B\in X$ such that for all transvections $A\in X$, the product $AB$ does not equal any nonidentity element $\lambda I$ of $Z(H)$. But this follows from the fact that $A$ and $B$ are unipotent (due to being transvections).
\end{example}

\begin{example}\label{ex:an}
    Let $n\geq 5$ be an integer, and let $X$ be any conjugacy class in the symmetric group $S_n$ satisfying one of the following conditions:
    \begin{itemize}
        \item $X$ contains odd permutations; that is, $X\cap A_n=\emptyset$.
        \item $X$ is a single conjugacy class in $A_n$; that is, $X$ does not split into two conjugacy classes in the alternating group $A_n$.
    \end{itemize}
    If $X$ does not contain involutions, then a straightforward verification of the conditions in Proposition \ref{prop:cns} shows that $(\Conj X,\iota|_X)$ is a CNS-quandle. In fact, since $S_n$ and $A_n$ are centerless, Corollary \ref{cor:centerless2} shows that $\iota|_X$ is the \emph{only} good involution of $\Conj X$.
    
    In particular, let $k$ be an integer such that $3\leq k\leq 2\lfloor n/2\rfloor$, and let $X$ be the set of $k$-cycles in $S_n$. Then we obtain the second part of Proposition \ref{prop:sl}.
\end{example}

\begin{table}[]
\caption{Unique sizes of conjugacy classes $X$ of $S_n$ that satisfy the conditions in Example \ref{ex:an}.}
\label{table:sn}
\begin{tabular}{c|c}
$n$ & Unique sizes of $X$                                                                                                                                                                   \\ \hline
$5$ & $20$, $30$                                                                                                                                                                     \\
$6$ & $40$, $90$, $120$                                                                                                                                                        \\
$7$ & $70$, $210$, $280$, $420$, $504$, $630$, $840$                                                                                                                                 \\
$8$ & $112$, $420$, $1120$, $1260$, $1344$, $1680$, $2520$, $3360$, $4032$, $5040$                                                                                                   \\
$9$ & \begin{tabular}[c]{@{}c@{}}$168$, $756$, $2240$, $2520$, $3024$, $3360$, $7560$, $9072$, $10080$,\\ $11340$, $15120$, $18144$, $20160$, $25920$, $30240$, $45360$\end{tabular}
\end{tabular}
\end{table}

Addressing Problem \ref{prob:cns2}, Table \ref{table:sn} displays the orders of the CNS-subquandles $(\Conj X,\iota|_X)$ of $(\Conj S_n,\iota)$ constructed in Example \ref{ex:an} for all $5\leq n\leq 9$, omitting duplicate orders.
    These orders were obtained using the \texttt{GAP} \cite{GAP4} program provided in Appendix \ref{app:a}.

\subsubsection{An infinite-order example}
While the previous examples have studied finite conjugation subquandles, our methods also apply in the infinite-order case.

\begin{example}\label{ex:su2}
    Let $G$ be the special unitary group $\SU(2)$, considered as a matrix group. We compute the good involutions of certain connected subquandles of $\Conj G$ called \emph{spherical quandles}, which have various knot-theoretic and representation-theoretic applications \citelist{\cite{su2}\cite{su2-2}}.

    It is well-known (see, for example, \cite{hall}*{Ex.\ 9 in Ch.\ 11, Sec.\ 8}) that the diagonal matrices \[d_\theta:=\operatorname{diag}(e^{i\theta},e^{-i\theta})\] with $\theta\in[0,\pi]$ represent all conjugacy classes of $G$, and an element $g\in G$ is conjugate to $d_\theta$ if and only if the eigenvalues of $g$ are $e^{\pm i\theta}$. Let $X_\theta$ denote the conjugacy class of $d_\theta$ in $G$.

    Fix $\theta\in (0,\pi)$. Since the only nontrivial normal subgroup of $G$ is the center
    \[
    Z(G)=\{\pm I_2\},
    \]
    it follows that $X_\theta$ generates $G$. Thus, Lemma \ref{lem:conn} shows that $\Conj X_\theta$ is connected. By the above discussion, $\Conj X_\theta$ is also inversion-closed. By verifying the conditions in Corollary \ref{cor:connected} and Proposition \ref{prop:cns}, we deduce that
    \[
    \Good(\Conj X_\theta)=\begin{cases}
        \{\iota|_{X_\theta}, \id_{X_\theta}\} & \text{if } \theta=\pi/2,\\ \{\iota|_{X_\theta}\} & \text{otherwise,}
    \end{cases}
    \]
    and $(\Conj(X_\theta),\iota|_{X_\theta})$ is a CNS-quandle if and only if $\theta\neq \pi/2$. In particular, $(\Conj G,\iota)$ contains uncountably many CNS-subquandles.
\end{example}

\section{Connections with Legendrian racks}\label{sec:leg}

In this section, we prove Theorem \ref{thm:legendrian} and several other equivalences of categories (Theorem \ref{thm:gl}). Before that, we introduce kink-involutory racks and discuss medial racks, GL-racks, and Legendrian racks. We also discuss a canonical rack automorphism $\theta$ that plays a fundamental role in the theory.

\subsection{Discussion of results}

The following is the main result of this section.

\begin{thm}\label{thm:gl}
    In each of the following pairs or triples, the categories mentioned are equivalent (in fact, isomorphic), and so are their full subcategories whose objects are medial:
    \begin{itemize}
        \item Racks $(\Rc)$ and Legendrian racks $(\LR)$.
        \item Racks for which the canonical automorphism $\theta$ is an involution $(\kir)$ and Legendrian quandles $(\LQ)$.
        \item Involutory GL-racks $(R,\uu)$ in which $\uu$ is an involution and involutory symmetric racks.
        \item Involutory racks $(\InvRk)$, Legendrian kei $(\LK)$, and symmetric kei $(\SK)$.
    \end{itemize}
\end{thm}

After we construct functors between $\Rc$ and $\LR$, Theorem \ref{thm:gl} will more or less follow from the definitions, all of which we provide in this section.

\subsubsection{Corollaries}
In tandem with \cite{taGL}*{Thm.\ 5.6}, Theorem \ref{thm:gl} yields the following.

\begin{cor}
    The category $\LR$ and the category of GL-quandles are isomorphic, and so are their full subcategories whose objects are medial.
\end{cor}

By combining Theorem \ref{thm:gl} with \cite{taGL}*{Thm.\ 5.1}, we deduce the center $Z(\SK)$ of the category of symmetric kei.

\begin{cor}
    The center of the category $\SK$ is the group generated by the collection of all good involutions: \[Z(\SK)=\langle\rho\mid\rho^2=1\rangle\cong\Z/2\Z.\]
\end{cor}

Since $\LK$ and $\SK$ are isomorphic, classification results in either category hold in the other. For example, Theorem \ref{thm:core} classifies all Legendrian structures on core quandles $\Core G$, and (by Observation \ref{obs:conj2}) Theorem \ref{thm:conj} classifies all Legendrian structures on subquandles $\Conj X$ of conjugation quandles $\Conj G$ such that $x^2\in Z(\langle X\rangle)$ for all $x\in X$. 

Conversely, \cite{taGL}*{Prop.\ 4.11} slightly refines a classification result of Kamada and Oshiro \cite{symm-quandles-2}*{Thm.\ 3.2} for dihedral quandles of even order. Note that Corollary \ref{cor:taka} recovers their result for dihedral quandles of odd order.

\begin{cor}
    Let $n\in\Z^+$ be a positive integer, and define affine transformations $\alpha_{m,b}:\Z/2n\Z\to\Z/2n\Z$ by $x\mapsto mx+b$. Let $R_{2n}=T(\Z/2n\Z)$ be the dihedral quandle of order $2n$. Then
    \[
    \Good R_{2n}=\begin{cases}
        \{\id_{\Z/2n\Z},\alpha_{1,n}\} & \text{if }n \text{ is odd,}\\
        \{\id_{\Z/2n\Z},\alpha_{1,n},\alpha_{n+1,0},\alpha_{n+1,n}\} & \text{otherwise.}
    \end{cases}
    \]
    Of the corresponding symmetric kei, the only isomorphic ones are \[(R_{2n},\alpha_{n+1,0})\cong (R_{2n},\alpha_{n+1,n})\]
    for even $n$. In particular, there are two isomorphism classes of symmetric kei of the form $(R_{2n},\rho)$ when $n$ is odd, and there are three when $n$ is even.
\end{cor}

\subsubsection{Computer search}

\begin{table}[]
\caption{Enumeration of isomorphism classes of various types of Legendrian racks up to order $8$.}
\label{table:gl}
\centering
\begin{tabular}{l|ccccccccc}
Order                        & 0 & 1 & 2 & 3 & 4  & 5  & 6   & 7    & 8     \\ \hline
Legendrian racks           & 1 & 1 & 2 & 6 & 19 & 74 & 353 & 2080 & 16023 \\
Medial Legendrian racks    & 1 & 1 & 2 & 6 & 18 & 68 & 329 & 1965 & 15455 \\
Legendrian quandles        & 1 & 1 & 2 & 5 & 15 & 54 & 240 & 1306 & 9477  \\
Medial Legendrian quandles & 1 & 1 & 2 & 5 & 14 & 48 & 219 & 1207 & 9042 \\
Legendrian kei & 1 & 1 & 2 & 5 & 13 & 42 & 180 & 906 & 6317 \\
Medial Legendrian kei & 1 & 1 & 2 & 5 & 12 & 38 & 168 & 850 & 6090 
\end{tabular}
\end{table}

Adding to our classification of symmetric racks up to order $8$ using \texttt{GAP} \cite{GAP4}, we used the classification of GL-racks up to order $8$ in \cite{taGL} to classify various types of Legendrian racks up to isomorphism. 

We provide our code and data in a GitHub repository \cite{gl-code}. Table \ref{table:gl} enumerates our data; cf.\ Table \ref{table:counts}, \cite{library}*{Tab.\ 1}, and \cite{taGL}*{Tab.\ A.1}. Note that the rows of Table \ref{table:gl} also count isomorphism classes of the other objects listed in Theorem \ref{thm:gl}, some of which were recorded in \cite{library}*{Tab.\ 1}.

\subsection{Preliminaries: Racks} We discuss the category $\Rc$ and several full subcategories of $\Rc$. See \citelist{\cite{rack-roll}\cite{center}\cite{joyce}} for further discussion of racks and quandles from a categorical perspective.

\subsubsection{The canonical automorphism $\theta$}
Every rack $R=(X,s)$ has a canonical automorphism $\theta_R$ defined by \[x\mapsto s_x(x)\] for all $x\in X$; see \cite{center}*{Prop.\ 2.5} and cf.\ \cite{taGL}*{Prop.\ 2.16}.\footnote{Some authors, especially those working in framed knot theory (e.g., \cite{quandlebook}*{p.\ 149}), call $\theta_R$ the \emph{kink map} of $R$ and denote it by $\pi$. This motivates our choice of nomenclature in Definition \ref{def:kink}.} Evidently, $R$ is a quandle if and only if $\theta_R=\id_X$, so we can loosely think of $\theta_R$ as the obstruction for $R$ to be a quandle. When there is no ambiguity, we will suppress the subscript and only write $\theta$ to mean $\theta_R$. 

\begin{example}
        If $(X,s)$ is a permutation rack defined by $\sigma\in S_X$, then $\theta=\sigma$.
    \end{example}

    Recall that the \emph{center} of a category $\mathcal{C}$ is the commutative monoid $Z(\mathcal{C})$ of natural endomorphisms of the identity functor $\mathbf{1}_\mathcal{C}$. 
Concretely, $\eta\in Z(\mathcal{C})$ if and only if, for all objects $R,S$ and morphisms $\varphi:R\to S$ in $\mathcal{C}$, the component $\eta_R$ is an endomorphism of $R$, and
\[
\eta_S\varphi=\varphi\eta_R.
\]

Our proof of Theorem \ref{thm:gl} uses the following result of Szymik \cite{center} in 2018.

\begin{prop}[\cite{center}*{Thm.\ 5.4}]
    The center $Z(\Rc)$ of the category of racks is the free group $\langle \theta\rangle\cong\Z$ generated by the collection $\theta$ of canonical automorphisms $\theta_R$ for all racks $R$.
\end{prop}

\subsubsection{Kink-involutory racks}
We introduce a class of racks called \emph{kink-involutory racks} that generalize quandles and involutory racks.

    \begin{definition}\label{def:kink}
            A rack $R=(X,s)$ is called \emph{kink-involutory} if its canonical automorphism $\theta$ is an involution.  
            Let $\kir$ be the full subcategory of $\Rc$ whose objects are kink-involutory.
    \end{definition}

    \begin{example}
        All quandles are kink-involutory.
    \end{example}

\begin{prop}\label{prop:inv-racks}
        All involutory racks are kink-involutory.
    \end{prop}

    \begin{proof}
        Let $R=(X,s)$ be an involutory rack. For all $x\in X$,
        \[
        \theta^2(x)=\theta s_x(x)=s_{s_x(x)}s_x(x)=s_x s_x(x)=x.
        \]
        Thus,  $\theta^2=\id_X$.
    \end{proof}
    
\begin{example}
            For an example of a kink-involutory rack that is neither involutory nor a quandle, consider the rack $R=(X,s)$ with underlying set $X=\{1,2,3,4,5\}$ whose rack structure is given by the permutations
        \[
        s_1=s_2=(12)(345),\quad\quad s_3=s_4=s_5=(12)
        \]
        in cycle notation. Then $s_1^2\neq \id_X$, so $R$ is not involutory. Also, $\theta=(12)$ is a nonidentity involution, so $R$ is kink-involutory and not a quandle.
        \end{example}
        
\subsubsection{Medial racks}
We discuss a class of racks called \emph{medial racks} or \emph{abelian racks}, which Joyce \cite{joyce} introduced to study Alexander invariants of knots.\footnote{In this paper, we adopt the name ``medial'' over ``abelian.'' This is to prevent confusion with \emph{commutative} racks, which satisfy the much rarer condition that $s_x(y)=s_y(x)$ for all $x,y\in X$.}  Certain categorical and universal-algebraic notions (see, for example, \cite{rack-roll}*{Sec.\ 3}, \cite{Hom}*{Thm.\ 12}, or \cite{taGL}*{Sec.\ 2.3.3}) allow medial racks to enhance coloring invariants of smooth knots \cite{Hom}*{Ex.\ 9} and Legendrian knots \cite{taGL}*{Cor.\ 3.21}.

    	\begin{definition}
		[\cite{rack-roll}*{Sec.\ 3}]\label{def:medial}
		Let $R=(X,s)$ be a rack, and consider the \emph{product rack} $R\times R$. We say that $R$ is \emph{medial} or \emph{abelian} if the map $X\times X\to X$ defined by \[(x,y)\mapsto s_y(x)\] is a rack homomorphism from $R\times R$ to $R$.
	\end{definition}

    \begin{rmk}
        Unpacking Definition \ref{def:medial} recovers the pointwise definition of mediality introduced by Joyce \cite{joyce}*{Def.\ 1.3}.
    \end{rmk}

    \begin{example}
        All permutation racks and Takasaki kei are medial.
    \end{example}

\subsection{Preliminaries: GL-racks}
We discuss GL-racks, which Karmakar et al.\ \cite{karmakarGL} and Kimura \cite{bi} independently introduced in 2023 to develop invariants of Legendrian knots in contact three-space. See \cite{taGL} for a development of the algebraic theory of GL-racks and \cite{he} for applications. 

While the following definition contrasts with the original definitions of Karmakar et al.\ and Kimura, their equivalence was proven in \cite{taGL}*{Prop.\ 3.12}.

    	\begin{definition}[\cite{taGL}]\label{def:gl-rack}
		Given a rack $R=(X,s)$, a \emph{GL-structure} on $R$ is a rack automorphism $\uu\in\Aut R$ such that $\uu s_x=s_x\uu$ for all $x\in X$.
		We call the pair $(R,\uu)$ a \emph{GL-rack}, \emph{generalized Legendrian rack}, or \emph{bi-Legendrian rack}.
        \end{definition}

    	\begin{example}[\cite{bi}*{Ex.\ 3.7}]
		Given a permutation rack $P=(X,s)$ defined by $\sigma\in S_X$, the set of GL-structures on $P$ is precisely the centralizer $C_{S_X}(\sigma)$. 
	\end{example}
	
	\begin{definition}\label{def:gl-hom}
		A \emph{GL-rack homomorphism} between two GL-racks $(R_1,\uu_1)$ and $(R_2,\uu_2)$ is a rack homomorphism 
        $\varphi$ from $R_1$ to $R_2$ that satisfies 
        \(\varphi\uu_1=\uu_2\varphi\). Let $\glr$ be the category of GL-racks.
	\end{definition}

    \subsubsection{Centrality of GL-structures}
    By Definition \ref{def:gl-hom}, all integer powers of GL-structures $\uu$ and canonical rack automorphisms $\theta$ lie in the categorical center $Z(\glr)$. In fact, $Z(\glr)\cong\Z^2$ is the free abelian group generated by these two collections of automorphisms; see \cite{taGL}*{Thm.\ 5.1}.

    \subsubsection{Legendrian racks} As their name suggests, \emph{Legendrian racks} are a special class of GL-racks defined below. They were introduced by Ceniceros et al.\ \cite{ceniceros} in 2021, and they generalize the constructions of Kulkarni and Prathamesh \cite{original}.

            Although the following definition differs from the original definition of Ceniceros et al., their equivalence was proven in \cite{taGL}*{Cor.\ 3.13}.

    \begin{definition}[\cite{taGL}]
        Let $(R,\uu)$ be a GL-rack. We say that $(R,\uu)$ is a \emph{Legendrian rack} if $\theta=\uu^{-2}$, in which case we say that $\uu$ is a \emph{Legendrian structure} on $R$. 
        
        We say that a Legendrian rack $(R,\uu)$ is a \emph{Legendrian quandle} if $R$ is a quandle or, equivalently, if $\uu$ is an involution. Similarly, we say that $(R,\uu)$ is a \emph{Legendrian kei} if $R$ is a kei.
        
        Let $\LR$, $\LQ$, and $\LK$ be the full subcategories of $\glr$ whose objects are Legendrian racks, Legendrian quandles, and Legendrian kei, respectively.
    \end{definition}    
    	\begin{example}[\cite{bi}*{Ex.\ 3.6}]
		Let $G$ be a group, and let $z\in Z(G)$ be a central element of $G$. 
        Then multiplication by $z$ defines a GL-structure on the conjugation quandle $\Conj G$. This GL-quandle is a Legendrian quandle if and only if $z^2=1$ in $G$.
	\end{example}

    The following part of Theorem \ref{thm:gl} follows immediately.

\begin{prop}\label{prop:equal}
    In each of the following pairs, the categories mentioned have precisely the same objects and morphisms:
    \begin{itemize}
        \item Involutory GL-racks $(R,\uu)$ in which $\uu$ is an involution and involutory symmetric racks.
        \item Legendrian kei $(\LK)$ and symmetric kei $(\SK)$.
    \end{itemize}
\end{prop}

\begin{proof}
    This follows directly from Proposition \ref{cor:inv-cent} and the definitions of $\glr$ and $\LK$.
\end{proof}

\subsection{Proof of Theorem \ref{thm:gl}}

To prove Theorem \ref{thm:gl}, we construct functors $F: \mathsf{LR}\to\mathsf{Rack}$ and $F\inv:\Rc\to\LR$ satisfying the following criteria; cf.\ \cite{taGL}*{Sec.\ 5}.
\begin{prop}\label{prop1}
    The functors $F$ and $F\inv$ are mutually inverse. 
    Moreover, both functors send medial objects to medial objects.
\end{prop}

\begin{prop}\label{prop3}
    $F(\LQ)=\kir$.
\end{prop}

\begin{prop}\label{prop4}
    $F(\LK)=\InvRk$.
\end{prop}

Once we define $F$ and $F\inv$, these three results will more or less follow from the definitions. Together with Proposition \ref{prop:equal}, these results yield Theorem \ref{thm:gl}.

\subsubsection{Construction of functors}

Define $F:\LR\to\Rc$ on objects by
\[
(X,s,\uu)\mapsto(X,\uu^3 s),
\]
where the rack structure $\uu^3s:X\to S_X$ is defined by $x\mapsto\uu^3 s_x$. Let $F$ preserve the underlying set maps of morphisms in $\LR$.

To construct an inverse functor $F\inv:\Rc\to\LR$, define $F\inv$ on objects by
\[
R=(X,s)\mapsto (X,\theta^{-3}_Rs,\theta_R),
\]
where the rack structure $\theta^{-3}_Rs:X\to S_X$ is defined by $x\mapsto \theta^{-3}_R s_x$. Let $F$ preserve the underlying set maps of morphisms in $\Rc$.

Using the inclusions $\uu^3\in Z(\glr)$ and $\theta,\theta^{-3}\in Z(\Rc)$, it is straightforward to verify that $F$ and $F\inv$ are covariant functors.

\subsubsection{Proofs of Propositions \ref{prop1}--\ref{prop4}}

\begin{proof}[Proof of Proposition \ref{prop1}]
    It is clear that $FF\inv=\mathbf{1}_\Rc$. Using the fact that $\theta_R=\uu^{-2}$ for all Legendrian racks $(R,\uu)$, verifying that $F\inv F=\mathbf{1}_\LR$ is straightforward. 
    Finally, the claim that $F$ and $F\inv$ send medial objects to medial objects follows from the fact that, for all GL-racks $(R,\uu)$, the maps $\theta_R^{-3}$ and $\uu$ are endomorphisms of $R$.
\end{proof}

\begin{proof}[Proof of Proposition \ref{prop3}]
    We verify that $F(\LQ)=\kir$. If $R=(X,s)$ is a kink-involutory rack, then it is straightforward to check that the underlying rack $(X,\theta^{-1}_Rs)$ of the Legendrian rack $F\inv(R)$ is a quandle.
    
    Conversely, if $L=(X,s,\uu)$ is a Legendrian quandle, then \[\theta_{F(L)}^2=(\uu^3\theta_{(X,s)})^2=\uu^2=\id_X,\] so  $F(L)$ is kink-involutory. 
\end{proof}

\begin{proof}[Proof of Proposition \ref{prop4}]
    We verify that $F(\LK)=\InvRk$. Let $L=(X,s,\uu)$ be a Legendrian kei. Since $\uu$ and $s_x$ commute and are involutions for all $x\in X$, the rack $F(L)=(X,\uu s)$ is involutory.
    
    Conversely, if $R=(X,s)$ is an involutory rack, then Proposition \ref{prop:inv-racks} and Proposition \ref{prop3} show that $F\inv(R)$ is a Legendrian quandle with rack structure $\theta_R^{-1}s$. Since $\theta^{-1}\in Z(\Rc)$, Proposition \ref{prop:inv-racks} shows that this rack structure is involutory. Hence, $F\inv(R)$ is a Legendrian kei.
\end{proof}

\section{Open questions}\label{sec:open}
We conclude by proposing directions for future work.

\begin{prob}
    Compute the set $\Anti R$ for various families of racks $R$.
\end{prob}

\begin{prob}
    Classify the group $\Aut R\cup\Anti R$ for various families of racks $R$.
\end{prob}

\begin{prob}
    Compute the set $\Good R$ for more families of racks $R$, including those that do not embed into conjugation quandles (cf.\ \citelist{\cite{embed}\cite{joyce}}).
\end{prob}

\begin{prob}
    Generalize Theorem \ref{thm2} to $n$-fold conjugation quandles (see \cite{matveev}*{Ex.\ 2 in Sec.\ 2}).
\end{prob}

\begin{prob}
    Distinguish the symmetric quandles obtained from Theorems \ref{thm2} and \ref{thm1} up to isomorphism.
\end{prob}

In light of our computational findings in Appendix \ref{app:b}, we also pose the following two questions.

\begin{prob}\label{open0}
    Let $G$ be a finite nonabelian group of order $n$. Is the number of good involutions of $\Conj G$ completely determined by $n$, the order of the center $Z(G)$, and the class number $k(G)$ of $G$? If so, is there an exact formula to compute this number?
\end{prob}

\begin{prob}\label{open1}
    In the setting of Corollary \ref{cor:orbits}, suppose that $|Z(H)|,|X/H|\geq 2$ (so, in particular, $\Conj X$ is not connected). Is the upper bound on $|{\Good(\Conj X)}|$ in Corollary \ref{cor:orbits} always strict?
\end{prob}

Note that if we drop either of the assumptions that $|Z(G)|,|X/H|\geq 2$, then Problem \ref{open1} has a negative answer; see Remark \ref{rmk:sharp}.

In light of Proposition \ref{prop:galkin} and Examples \ref{ex:sl25}--\ref{ex:su2}, we conclude with the following conjecture.

\begin{conj}\label{conjecture}
    Let $R$ be a connected quandle. If $R$ is noninvolutory, then $R$ has at most one good involution.
\end{conj}

By Observation \ref{obs:conj2} and Corollary \ref{cor:connected}, the specialization of Conjecture \ref{conjecture} to subquandles $\Conj X$ of conjugation quandles $\Conj G$ is equivalent to the following group-theoretic conjecture.

\begin{conj}
    Let $G$ be a group, and let $X$ be a subset of $G$ such that $X$ is a conjugacy class of the subgroup $H:=\langle X\rangle\leq G$. If there exist distinct central elements $z_1,z_2\in Z(H)$ such that $z_1x\inv,z_2x\inv\in X$ for all $x\in X$, then $x^2\in Z(H)$ for all $x\in X$.
\end{conj}

\bibliographystyle{amsplain}

\addresseshere

\clearpage

\appendix

\section{GAP code for Table \ref{table:sn}}\label{app:a}

Given an integer $n\geq 5$, the following \texttt{GAP} \cite{GAP4} program produces the orders of the CNS-subquandles of $\Conj S_n$ constructed in Example \ref{ex:an}. Table \ref{table:sn} displays the output for $5\leq n\leq 9$.

\begin{lstlisting}[language=GAP]
ComputeClassSizes := function(n)
    local Sn, An, classes, sizes, C, rep, ord, centSn, centAn;

    Sn := SymmetricGroup(n);
    An := AlternatingGroup(n);
    classes := ConjugacyClasses(Sn);
    sizes := [];

    for C in classes do
        rep := Representative(C);
        ord := Order(rep);
        if ord > 2 then
            # Use the orbit-stabilizer theorem to compare class sizes
            centSn := Centralizer(Sn, rep);
            centAn := Centralizer(An, rep);
            # Verify that C isn't a conjugacy class that splits in An
            if SignPerm(rep) = -1 or Size(centSn) = 2 * Size(centAn)
                then Add(sizes, Size(C));
            fi;
        fi;
    od;
    Sort(sizes);
    return sizes;
end;

for n in [5..9] do Print(ComputeClassSizes(n), "\n"); od;
\end{lstlisting}

\section{Census of good involutions}\label{app:b}

In this section, we enumerate good involutions of nontrivial conjugation quandles $\Conj G$ and nontrivial core quandles $\Core G$ of groups $G$ up to order $23$. 

Our enumerations are based on the computational data obtained from our \texttt{GAP} \cite{GAP4} implementations of Algorithms \ref{alg} and \ref{alg2}. We provide our code and the raw data (including the explicit mappings $\rho:G\to G$) in a GitHub repository \cite{code}.

\begin{notation} In Tables \ref{tab:app1} and \ref{tab:app2}, we use the following notation for groups. Let $n\geq 3$ be an integer.
    \begin{itemize}
        \item Given a group $G$, let $k(G)$ be the number of conjugacy classes of $G$.
        \item Let $S_n$ and $A_n$ denote the symmetric and alternating groups on $n$ letters, respectively.
        \item Let $D_n$ and $\Dic_n$ denote the dihedral group of order $2n$ and the dicyclic group of order $4n$, respectively.
        \item Let $Q_8$, $QD_{16}$, and $M_{16}$ denote the quaternion group of order $8$, the quasidihedral group of order $16$, and the modular maximal-cyclic group of order $16$, respectively.
    \end{itemize}
\end{notation}

In Table \ref{tab:app1}, for all nonabelian groups $G$ up to order $22$, we compare the number of good involutions of $\Conj G$ against the upper bound on this number from Corollary \ref{cor:orbits}. These results motivate Problems \ref{open0} and \ref{open1}. We only consider nonabelian groups because conjugation quandles of abelian groups are trivial. Cf.\ Table \ref{table:conj}.

\begin{table}[]
\caption{Number of good involutions of nontrivial conjugation quandles $\Conj G$ of groups $G$ up to order $22$, compared against the upper bound from Corollary \ref{cor:orbits}.}
\label{tab:app1}
\begin{tabular}{cccc}
Group                                & Order & $|{\Good(\Conj G)}|$ & $|Z(G)|^{k(G)}$ \\ \hline
$S_3$                                & 6     & 1                    & 1                     \\
$D_4$                                & 8     & 16                   & 32                    \\
$Q_8$                                & 8     & 16                   & 32                    \\
$D_5$                                & 10    & 1                    & 1                     \\
$\Dic_{3}$                           & 12    & 8                    & 64                    \\
$A_4$                                & 12    & 1                    & 1                     \\
$D_6$                                & 12    & 8                    & 64                    \\
$D_7$                                & 14    & 1                    & 1                     \\
$(\Z/2\Z\times\Z/2\Z)\rtimes \Z/4\Z$ & 16    & 2160                 & 1048576               \\
$\Z/4\Z\rtimes \Z/4\Z$               & 16    & 2160                 & 1048576               \\
$M_{16}$               & 16    & 2160                 & 1048576               \\
$D_8$                                & 16    & 32                   & 128                   \\
$QD_{16}$                            & 16    & 32                   & 128                   \\
$\Dic_{4}$                           & 16    & 32                   & 128                   \\
$D_4\times \Z/2\Z$                   & 16    & 2160                 & 1048576               \\
$Q_8\times \Z/2\Z$                   & 16    & 2160                 & 1048576               \\
$(\Z/4\Z\times\Z/2\Z)\rtimes \Z/2\Z$ & 16    & 2160                 & 1048576               \\
$D_9$                                & 18    & 1                    & 1                     \\
$S_3\times\Z/3\Z$                   & 18    & 64                   & 19683                 \\
$(\Z/3\Z\times\Z/3\Z)\rtimes \Z/2\Z$ & 18    & 1                    & 1                     \\
$\Dic_{5}$                           & 20    & 16                   & 256                   \\
$\Z/5\Z\rtimes\Z/4\Z$                & 20    & 1                    & 1                     \\
$D_{10}$                             & 20    & 16                   & 256                   \\
$\Z/7\Z\rtimes\Z/3\Z$                & 21    & 1                    & 1                     \\
$D_{11}$                             & 22    & 1                    & 1                     
\end{tabular}
\end{table}

In Table \ref{tab:app2}, for all groups $G$ of exponent greater than $2$ up to order $23$, we compare the number of good involutions of $\Core G$ against the upper bound on this number from Corollary \ref{ineq:core}. We do not consider elementary abelian $2$-groups because their core quandles are trivial. Cf.\ Table \ref{table:core}.

In both tables, we list groups in the order of their \texttt{GAP} implementations \texttt{SmallGroup(n,k)}, where \texttt{n} is the order of the group and \texttt{k} is the index of the group in the list \texttt{AllSmallGroups(n)}.

\clearpage

\begin{longtable}{cccc}
\caption{Number of good involutions of nontrivial core quandles $\Core G$ of groups $G$ up to order $23$, compared against the upper bound from Corollary \ref{ineq:core}.}
\label{tab:app2} \\

\multicolumn{1}{c}{Group} & \multicolumn{1}{c}{Order} & \multicolumn{1}{c}{$|{\Good(\Core G)}|$} & \multicolumn{1}{c}{$|T|^{|\oo|}$}\\ \hline 
\endfirsthead

\caption[]{(continued)}\\
\multicolumn{4}{c}%
{}\\
\multicolumn{1}{c}{Group} & \multicolumn{1}{c}{Order} & \multicolumn{1}{c}{$|{\Good(\Core G)}|$} & \multicolumn{1}{c}{$|T|^{|\oo|}$}\\ \hline 
\endhead
\endfoot\endlastfoot

$\Z/3\Z$                                & 3     & 1                    & 1                     \\
$\Z/4\Z$                                & 4     & 4                    & 4                     \\
$\Z/5\Z$                                & 5     & 1                    & 1                     \\
$S_3$                                & 6     & 1                    & 1                     \\
$\Z/6\Z$                                & 6     & 2                    & 4                     \\
$\Z/7\Z$                                & 7     & 1                    & 1                     \\
$\Z/8\Z$                                & 8     & 4                    & 4                     \\
$\Z/4\Z\times\Z/2\Z$                                & 8     & 36                    & 256                     \\
$D_4$                                & 8     & 16                   & 32                    \\
$Q_8$                                & 8     & 16                   & 16                    \\
$\Z/9\Z$                                & 9     & 1                    & 1                     \\
$\Z/3\Z\times\Z/3\Z$                                & 9     & 1                    & 1                     \\
$D_5$                                & 10    & 1                    & 1                     \\
$\Z/10\Z$                                & 10    & 2                    & 4                     \\
$\Z/11\Z$                                & 11     & 1                    & 1                     \\
$\Dic_{3}$                           & 12    & 8                    & 8                    \\
$\Z/12\Z$                                & 12     & 4                    & 4                     \\
$A_4$                                & 12    & 1                    & 1                     \\
$D_6$                                & 12    & 8                    & 64                    \\
$\Z/6\Z\times\Z/2\Z$                                & 12     & 10                    & 256                     \\
$\Z/13\Z$                                & 13     & 1                    & 1                     \\
$D_7$                                & 14    & 1                    & 1                     \\
$\Z/14\Z$                                & 14     & 2                    & 4                     \\
$\Z/15\Z$                                & 15     & 1                    & 1                     \\
$\Z/16\Z$                                & 16     & 4                    & 4                     \\
$\Z/4\Z\times\Z/4\Z$                                & 16     & 256                    & 256                     \\
$(\Z/2\Z\times\Z/2\Z)\rtimes \Z/4\Z$ & 16    & 576                 & 4096               \\
$\Z/4\Z\rtimes \Z/4\Z$               & 16    & 384                 & 1024               \\
$\Z/8\Z\times \Z/2\Z$               & 16    & 36                 & 256               \\
$M_{16}$               & 16    & 16                 & 16               \\
$D_8$                                & 16    & 32                   & 32                   \\
$QD_{16}$                            & 16    & 32                   & 32                   \\
$\Dic_{4}$                           & 16    & 32                   & 32                   \\
$\Z/4\Z\times\Z/2\Z\times \Z/2\Z$ & 16    & 5776                 & 16777216               \\
$D_4\times \Z/2\Z$                   & 16    & 2160                 & 1048576               \\
$Q_8\times \Z/2\Z$                   & 16    & 1296                 & 65536               \\
$(\Z/4\Z\times\Z/2\Z)\rtimes \Z/2\Z$ & 16    & 256                 & 256               \\
$\Z/17\Z$                                & 17     & 1                    & 1                     \\
$D_9$                                & 18    & 1                    & 1                     \\
$\Z/18\Z$                                & 18     & 2                    & 4                     \\
$S_3\times\Z/3\Z$                   & 18    & 1                   & 1                 \\
$(\Z/3\Z\times\Z/3\Z)\rtimes \Z/2\Z$ & 18    & 1                    & 1                     \\
$\Z/6\Z\times\Z/3\Z$                   & 18    & 2                   & 4                 \\
$\Z/19\Z$                                & 19     & 1                    & 1                     \\
$\Dic_{5}$                           & 20    & 16                   & 16                   \\
$\Z/20\Z$                                & 20     & 4                    & 4                     \\
$\Z/5\Z\rtimes\Z/4\Z$                & 20    & 1                    & 1                     \\
$D_{10}$                             & 20    & 16                   & 256                   \\
$\Z/10\times\Z/2\Z$                   & 20    & 10                   & 256                 \\
$\Z/7\Z\rtimes\Z/3\Z$                & 21    & 1                    & 1                     \\
$\Z/21\Z$                                & 21     & 1                    & 1                     \\
$D_{11}$                             & 22    & 1                    & 1                     \\
$\Z/22\Z$                                & 22     & 2                    & 4                     \\
$\Z/23\Z$                                & 23     & 1                    & 1                     
\end{longtable}

\end{document}